\documentclass[preprint,11pt,a4paper]{elsarticle}
\usepackage[a4paper,margin=25mm]{geometry}
\usepackage{amsmath,amssymb,amsthm,mathrsfs}
\usepackage{lmodern,microtype}
\usepackage[hidelinks]{hyperref}
\newtheorem{theorem}{Theorem}[section]
\newtheorem{lemma}[theorem]{Lemma}
\newtheorem{proposition}[theorem]{Proposition}
\newtheorem{corollary}[theorem]{Corollary}

\numberwithin{equation}{section}

\begin{document}
\setlength{\abovedisplayskip}{8pt plus 2pt minus 2pt}
\setlength{\belowdisplayskip}{8pt plus 2pt minus 2pt}
\setlength{\abovedisplayshortskip}{3pt plus 2pt}
\setlength{\belowdisplayshortskip}{6pt plus 2pt minus 2pt}
\begin{frontmatter}
\title{On the product of nontrivial cross-intersecting families}
\author[BW]{Biao Wu}
\ead{wu@hunnu.edu.cn}
\author[zh]{Huajun Zhang\corref{Y}}
\ead{huajunzhang@zjnu.cn}
\address[BW]{School of Mathematics and Statistics,
Hunan Normal University, Changsha 410081, P. R. China}
\address[zh]{Department of Mathematics,
Shaoxing University, Shaoxing 312000, P. R. China}
\cortext[Y]{Corresponding author.}

\begin{abstract}
Let $n\geq2\ell\geq2k\geq6$ and $n>k+\ell$.
We determine the maximum of $|\mathcal A||\mathcal B|$ over all
nonempty, nontrivial cross-intersecting families
$\mathcal A\subseteq\binom{[n]}k$ and
$\mathcal B\subseteq\binom{[n]}\ell$.
The exact bound is $|\mathcal A||\mathcal B|\leq\max\{M_1,M_2\}$, where
\[
\begin{aligned}
M_1
&=\left(\binom{n-1}{k-1}-\binom{n-\ell-1}{k-1}+1\right)
\left(\binom{n-1}{\ell-1}-\binom{n-k-1}{\ell-1}+1\right),\\
M_2
&=\left(\binom{n-2}{\ell-2}+1\right)
\sum_{i=1}^{2}
\left(\binom{n-i}{k-1}-\binom{n-\ell-i}{k-1}\right).
\end{aligned}
\]
Both terms are attained by explicit constructions: a
Hilton--Milner-type pair and a two-point construction, respectively.
We classify all equality cases, including the additional
cross-triangle pairs when $k=\ell=3$.
Two structural results underlie the proof.
First, we establish a restricted-lex comparison theorem showing that
restricted lexicographic initial segments minimize disjointness
neighborhoods, and characterize equality at the first nontrivial
threshold.
Second, a decomposition of critically shifted families into two
parts with disjoint cores leads to a sharp product theorem, with
equality characterized, for all cross-intersecting pairs admitting
this decomposition, without additional nontriviality assumptions.
Endpoint comparisons and reverse shifting complete the proof for
unequal as well as equal uniformities.
\end{abstract}
\end{frontmatter}

\section{Introduction}
For a positive integer $n$, write $[n]=\{1,\ldots,n\}$ and
$[a,b]=\{a,a+1,\ldots,b\}$, with $[a,b]=\emptyset$ if $a>b$.
For a finite set $X$, let $\binom Xk$ denote the family of all
$k$-element subsets of $X$. We use the convention
$\binom ab=0$ when $a\geq0$ and either $b<0$ or $b>a$.
A nonempty family $\mathcal F$ is \emph{intersecting} if any two of
its members intersect, and \emph{nontrivial} if
$\bigcap_{F\in\mathcal F}F=\emptyset$.

The Erd\H{o}s--Ko--Rado theorem is a fundamental result in extremal
set theory.
\begin{theorem}[Erd\H{o}s--Ko--Rado \cite{EKR}]
Let $n\geq2k$. If $\mathcal F\subseteq\binom{[n]}k$ is intersecting,
then $|\mathcal F|\leq\binom{n-1}{k-1}$. If $n>2k$, equality holds
if and only if $\mathcal F=\{F\in\binom{[n]}k:x\in F\}$ for some
$x\in[n]$.
\end{theorem}
The extremal family in this theorem is a full star. Excluding stars
leads to the Hilton--Milner theorem. Define
\begin{align*}
\mathcal H(n,k,1)&=
\{F\in\tbinom{[n]}k:1\in F,\ F\cap[2,k+1]\neq\emptyset\}
\cup\{[2,k+1]\},\\
h(n,k,1)&=\binom{n-1}{k-1}-\binom{n-k-1}{k-1}+1.
\end{align*}
\begin{theorem}[Hilton--Milner \cite{HM1967}]
Let $n>2k$ and $k\geq2$. If $\mathcal F\subseteq\binom{[n]}k$ is
nontrivial and intersecting, then $|\mathcal F|\leq h(n,k,1)$.
\end{theorem}

Two families $\mathcal A$ and $\mathcal B$ are \emph{cross-intersecting}
if $A\cap B\neq\emptyset$ for every $A\in\mathcal A$ and
$B\in\mathcal B$. Pyber \cite{Pyber} considered the product of the
sizes of two cross-intersecting families of the same uniformity.
Matsumoto and Tokushige established the following extension.
\begin{theorem}[Matsumoto--Tokushige \cite{MT1989}]\label{ckk}
Let $n\geq2\ell\geq2k$. If
$\mathcal A\subseteq\binom{[n]}k$ and
$\mathcal B\subseteq\binom{[n]}\ell$ are cross-intersecting, then
\[
|\mathcal A||\mathcal B|
\leq\binom{n-1}{k-1}\binom{n-1}{\ell-1}.
\]
\end{theorem}
We consider the corresponding problem when both families are
nontrivial. For equal uniformities, Frankl and Wang proved the
Hilton--Milner product bound for $k\geq8$ and $n\geq4k$
\cite{FW2023}, and subsequently extended it to $n\geq2k+1$
\cite{FWII}. Frankl \cite{Frankl3Graphs2026} settled the case $k=3$.
Huang \cite{Huang2026} established the bound for
$k\geq3$ and $n\geq2k$. For unequal uniformities, constructions
exceeding the natural Hilton--Milner product were given in
\cite{Huang2026,Bu2026}. The sharp bound therefore requires an
additional extremal construction.

For $n\geq2\ell\geq2k\geq6$, with $n>k+\ell$, set
\begin{align*}
M_1(n,k,\ell)
&=\left(\binom{n-1}{k-1}-\binom{n-\ell-1}{k-1}+1\right)
  \left(\binom{n-1}{\ell-1}-\binom{n-k-1}{\ell-1}+1\right),\\
M_2(n,k,\ell)
&=\left(\binom{n-2}{\ell-2}+1\right)
  \sum_{i=1}^{2}\left(\binom{n-i}{k-1}-\binom{n-\ell-i}{k-1}\right).
\end{align*}
Both quantities arise from nontrivial cross-intersecting pairs.
For $M_1$, choose $X\in\binom{[2,n]}k$ and
$Y\in\binom{[2,n]}\ell$ with $X\cap Y\neq\emptyset$, and take
\begin{align*}
\mathcal A&=\{A\in\tbinom{[n]}k:1\in A,\ A\cap Y\neq\emptyset\}
                \cup\{X\},\\
\mathcal B&=\{B\in\tbinom{[n]}\ell:1\in B,\ B\cap X\neq\emptyset\}
                \cup\{Y\}.
\end{align*}
Cross-intersection follows directly from the definitions. Each family
is nontrivial: its exceptional member avoids $1$, and, for every
$x\neq1$, its star part contains a member avoiding $x$.
Counting the star parts gives product $M_1(n,k,\ell)$.
For $M_2$, take
\begin{align*}
\mathcal A&=\{A\in\tbinom{[n]}k:
 A\cap[2]\neq\emptyset,\ A\cap[3,\ell+2]\neq\emptyset\},\\
\mathcal B&=\{B\in\tbinom{[n]}\ell:[2]\subseteq B\}
                  \cup\{[3,\ell+2]\}.
\end{align*}
These families are also nontrivial and cross-intersecting.
Proposition~\ref{prop:two-point-construction} gives the counting
argument in greater generality.

The sharp product bound is as follows.
\begin{theorem}\label{mthm}
Let $n\geq2\ell\geq2k\geq6$, with $n>k+\ell$.
If $\mathcal A\subseteq\binom{[n]}k$ and
$\mathcal B\subseteq\binom{[n]}\ell$ are nonempty, nontrivial
cross-intersecting families, then
\[
 |\mathcal A||\mathcal B|
 \leq\max\{M_1(n,k,\ell),M_2(n,k,\ell)\}.
\]
Equality holds if and only if one of the following holds:
\begin{enumerate}
\item[\rm(i)] $M_1(n,k,\ell)\geq M_2(n,k,\ell)$, and there are
$x\in[n]$, $X\in\binom{[n]\setminus\{x\}}k$ and
$Y\in\binom{[n]\setminus\{x\}}\ell$ with
$X\cap Y\neq\emptyset$ such that
\begin{align*}
 \mathcal A&=\{A\in\tbinom{[n]}k:x\in A,\ A\cap Y\neq\emptyset\}
                   \cup\{X\},\\
 \mathcal B&=\{B\in\tbinom{[n]}\ell:x\in B,\ B\cap X\neq\emptyset\}
                   \cup\{Y\}.
\end{align*}
\item[\rm(ii)] $M_2(n,k,\ell)\geq M_1(n,k,\ell)$, and there are
disjoint sets $R\in\binom{[n]}2$ and
$Y\in\binom{[n]}\ell$ such that
\begin{align*}
 \mathcal A&=\{A\in\tbinom{[n]}k:
                  A\cap R\neq\emptyset,\ A\cap Y\neq\emptyset\},\\
 \mathcal B&=\{B\in\tbinom{[n]}\ell:R\subseteq B\}\cup\{Y\}.
\end{align*}
\item[\rm(iii)] $k=\ell=3$, and there are $x\in[n]$ and
$P,Q\in\binom{[n]\setminus\{x\}}2$ with
$P\cap Q\neq\emptyset$ such that
\begin{align*}
 \mathcal A&=\{A\in\tbinom{[n]}3:x\in A,\ A\cap Q\neq\emptyset\}
       \cup\{A\in\tbinom{[n]}3:x\notin A,\ P\subseteq A\},\\
 \mathcal B&=\{B\in\tbinom{[n]}3:x\in B,\ B\cap P\neq\emptyset\}
       \cup\{B\in\tbinom{[n]}3:x\notin B,\ Q\subseteq B\}.
\end{align*}
\end{enumerate}
When $M_1(n,k,\ell)=M_2(n,k,\ell)$, both (i) and (ii)
are included. In (iii), the two pairs $P,Q$ need not coincide.
\end{theorem}
In Section~\ref{sec:shifting}, we establish the shifting reduction
and prove the restricted-lex comparison directly by induction
on the last nonempty layer, using column exchanges and Pascal
refinement in an upper triangular matrix. Section~\ref{sec:critical-revised} proves a sharp
product bound, with equality characterized, for families admitting
two disjoint cores (Theorem~\ref{thm:two-core-product}); the
two-core decomposition of critically shifted families then yields
Corollary~\ref{critically} and settles the critical case of
Theorem~\ref{mthm}.
Section~\ref{sec:shifted-product} applies the restricted-lex
comparison to shifted pairs and completes the proof, including
all equality cases. At the threshold, we use F\"uredi's
classification~\cite{Furedi1995} for $k+\ell\geq7$ and the
cross-triangle classification in~\cite{Huang2026} for $k=\ell=3$.
The distinction is necessary: the pairs $P,Q$ in
Theorem~\ref{mthm}(iii) need not be equal.

\section{Shifting and lexicographic reductions}\label{sec:shifting}
\subsection{Sections, neighborhoods, and simultaneous shifts}
For disjoint $P,Q\subseteq[n]$ and a family $\mathcal F$, define
\begin{align*}
\mathcal F(P,\overline Q)
 &=\{F\setminus P:F\in\mathcal F,\ F\cap(P\cup Q)=P\},\\
\mathcal F\langle P,\overline Q\rangle
 &=\{F\in\mathcal F:F\cap(P\cup Q)=P\},\\
\mathcal F(P,\widetilde Q)
 &=\{F\setminus P:F\in\mathcal F,\ P\subseteq F,\ Q\nsubseteq F\},\\
\mathcal F\langle P,\widetilde Q\rangle
 &=\{F\in\mathcal F:P\subseteq F,\ Q\nsubseteq F\}.
\end{align*}
We omit an empty restriction and write, for example,
$\mathcal F(P)=\mathcal F(P,\overline\emptyset)$ and
$\mathcal F(\overline Q)=\mathcal F(\emptyset,\overline Q)$.
Strings such as $12$ denote sets of distinct displayed elements.
Round brackets remove the prescribed elements; angle brackets retain
them.

For a family $\mathcal F\subseteq2^U$, define its disjointness
neighborhood and its cross-intersecting companion by
\begin{align*}
N_q^U(\mathcal F)
 &=\{G\in\tbinom Uq:G\cap F=\emptyset\text{ for some }F\in\mathcal F\},\\
\overline N_q^U(\mathcal F)&=\binom Uq\setminus N_q^U(\mathcal F).
\end{align*}
We omit $U$ when $U=[n]$. In particular, membership in $N_q^U$
always refers to disjointness. A pair is cross-intersecting if and
only if $\mathcal B\subseteq\overline N_\ell(\mathcal A)$.
A cross-intersecting pair is called \emph{saturated} if
each family is the full companion of the other:
\begin{equation}\label{eq:saturation}
\mathcal A=\overline N_k(\mathcal B),\qquad
\mathcal B=\overline N_\ell(\mathcal A).
\end{equation}
Every nontrivial cross-intersecting pair of maximum product is
saturated: adjoining an admissible member preserves nontriviality and
strictly increases the product when the other family is nonempty.

For $1\leq i<j\leq n$ and $F\in\mathcal F$, set
\[
s_{ij}^{\mathcal F}(F)=
\begin{cases}
(F\setminus\{j\})\cup\{i\},&
 j\in F,\ i\notin F,\ (F\setminus\{j\})\cup\{i\}\notin\mathcal F,\\
F,&\text{otherwise},
\end{cases}
\]
and let $s_{ij}(\mathcal F)=\{s_{ij}^{\mathcal F}(F):F\in\mathcal F\}$.
Shifting preserves cardinality and uniformity; simultaneous shifting
preserves cross-intersection \cite{Fra1987}. A family is
\emph{shifted on $D$} if it is fixed by every $s_{ij}$ with
$i,j\in D$ and $i<j$. It is \emph{shifted} if $D=[n]$.

A nontrivial cross-intersecting pair $(\mathcal A,\mathcal B)$ is
\emph{critically shifted} if at least one family is not shifted and,
for every $i<j$, either both families are fixed by $s_{ij}$ or at
least one of $s_{ij}(\mathcal A),s_{ij}(\mathcal B)$ is trivial.
This is a property of the pair. It does not specify which family
becomes trivial under a particular shift.
\begin{lemma}[Normalization]\label{lem:normalization}
Any nontrivial cross-intersecting pair can be transformed, without
changing either cardinality, into a shifted pair or a critically
shifted pair. Moreover, a product-maximizing pair can be chosen to
satisfy \eqref{eq:saturation} and this alternative.
\end{lemma}
\begin{proof}
Apply a simultaneous shift whenever it changes at least one family
and preserves the nontriviality of both. Each such step strictly
decreases the integer weight
$\sum_{A\in\mathcal A}\sum_{x\in A}x+
 \sum_{B\in\mathcal B}\sum_{x\in B}x$.
The procedure therefore terminates, and its terminal pair has the
required property. If the original pair maximizes the product, every
pair in the procedure also does so and is saturated by
\eqref{eq:saturation}.
\end{proof}

\subsection{The Kruskal--Katona theorem and its dual form}
For distinct sets $E,F$ of the same size, define the two orders by
\[
E<_{\rm lex}F\iff\min(E\triangle F)\in E,\qquad
E<_{\rm colex}F\iff\max(E\triangle F)\in F.
\]
Let $\mathcal L(U,k,m)$ be the first $m$
$k$-subsets of the ordered set $U$ in lexicographic order.
For a $p$-uniform family $\mathcal G$, put
$\partial^0\mathcal G=\mathcal G$ and
$\partial\mathcal G=\{H:|H|=p-1,\ H\subseteq G
\text{ for some }G\in\mathcal G\}$;
iterate this operation to define $\partial^d\mathcal G$.

Every positive integer $m$ has a unique $p$-binomial expansion
\[
m=\sum_{j=t}^{p}\binom{b_j}{j},\qquad
b_p>b_{p-1}>\cdots>b_t\geq t\geq1.
\]
\begin{theorem}[Kruskal--Katona \cite{Kruskal,Katona68}]\label{thm:KK}
If $|\mathcal G|=m$ and $\mathcal G$ is $p$-uniform, then, for
$0\leq d\leq p$,
\[
|\partial^d\mathcal G|\geq
\sum_{j=t}^{p}\binom{b_j}{j-d}.
\]
Equality is attained by an initial segment in colexicographic order.
For $1\leq d<p$, M\"ors' equality criterion
\cite[Theorem~7]{Mors1985} gives uniqueness up to isomorphism if
\[
t>d,\qquad\text{or}\qquad
m=\binom{b_p+1}{p}-1,\qquad\text{or}\qquad m\leq p+1.
\]
\end{theorem}
\begin{theorem}[Dual Kruskal--Katona theorem]\label{cokk}
Let $n\geq k+\ell$ and $\mathcal A\subseteq\binom{[n]}k$ be
nonempty. Write $m=|\mathcal A|$ in its $(n-k)$-binomial expansion as
$m=\sum_{j=t}^{n-k}\binom{b_j}{j}$. Then
\[
|N_\ell(\mathcal A)|\geq
\sum_{j=t}^{n-k}\binom{b_j}{j-(n-k-\ell)}
=|N_\ell(\mathcal L([n],k,|\mathcal A|))|.
\]
If $n>k+\ell$ and at least one of
\[
t>n-k-\ell,\qquad
m=\binom{b_{n-k}+1}{n-k}-1,\qquad m\leq n-k+1
\]
holds, equality holds if and only if
$\mathcal A\cong\mathcal L([n],k,m)$.
\end{theorem}
\begin{proof}
Let $\mathcal A^c=\{[n]\setminus A:A\in\mathcal A\}$.
Then $N_\ell(\mathcal A)=\partial^{n-k-\ell}(\mathcal A^c)$,
so the inequality follows from Kruskal--Katona.
Complementation followed by the relabeling $i\mapsto n+1-i$ maps
an initial lexicographic segment of $k$-sets onto an initial
colexicographic segment of $(n-k)$-sets. This proves the equality
asserted for $\mathcal L([n],k,|\mathcal A|)$.
The uniqueness statement follows from M\"ors' criterion in
Theorem~\ref{thm:KK}, applied to $\mathcal A^c$ with
$p=n-k$ and $d=n-k-\ell$.
\end{proof}
\begin{theorem}[Hilton's lemma \cite{Hil1976}]\label{lem1}
If $\mathcal A\subseteq\binom{[n]}k$ and
$\mathcal B\subseteq\binom{[n]}\ell$ are cross-intersecting, then
$\mathcal L([n],k,|\mathcal A|)$ and
$\mathcal L([n],\ell,|\mathcal B|)$ are cross-intersecting.
\end{theorem}
Lexicographic replacement need not preserve nontriviality. Thus
Hilton's lemma alone does not justify its use within the class of
nontrivial pairs.

\subsection{The restricted universe and the layer decomposition}
Throughout the remainder of this section, assume $n>k+\ell$ and
$k,\ell\geq2$.
Set
\[
\mathcal B(n,k,\ell)=
\{A\in\tbinom{[n]}k:|A\cap[\ell+1]|\geq2\},\qquad
S=\binom{n-1}{k-1}-\binom{n-\ell-1}{k-1}.
\]
The restricted universe is shifted, and
\[
|\mathcal B(n,k,\ell)|
=\sum_{j=2}^{\min\{k,\ell+1\}}
\binom{\ell+1}j\binom{n-\ell-1}{k-j}.
\]
For $0\leq m\leq|\mathcal B(n,k,\ell)|$, let
$\mathcal L_\ell([n],k,m)$ denote its first $m$ members in
lexicographic order.
\begin{lemma}\label{lem:restricted-universe}
Let $n>k+\ell$ and $k,\ell\geq2$.
\begin{enumerate}
\item[\rm(i)] If $\mathcal A\subseteq\binom{[n]}k$ is shifted and
$\overline N_\ell(\mathcal A)$ is nonempty and nontrivial, then
$\mathcal A\subseteq\mathcal B(n,k,\ell)$.
\item[\rm(ii)] If $\mathcal A\subseteq\mathcal B(n,k,\ell)$ and
$|\mathcal A|>S$, then both $\mathcal A$ and
$\overline N_\ell(\mathcal A)$ are nontrivial, and the latter is
nonempty.
\end{enumerate}
\end{lemma}
\begin{proof}
The companion of a shifted family is shifted. To see this, let
$B\in\overline N_\ell(\mathcal A)$, $j\in B$, $i\notin B$ and
$i<j$, and put $B'=(B\setminus\{j\})\cup\{i\}$.
If $A\in\mathcal A$ were disjoint from $B'$, then
$A\cap B=\{j\}$ and $i\notin A$. Shiftedness would give
$(A\setminus\{j\})\cup\{i\}\in\mathcal A$, disjoint from $B$,
a contradiction.

Any shifted nontrivial $\ell$-uniform family contains a member
avoiding $1$, hence contains $[2,\ell+1]$ and all of its left
shifts. It therefore contains $\binom{[\ell+1]}\ell$.
Every member of $\mathcal A$ must meet all these sets, which is
equivalent to having at least two elements in $[\ell+1]$.
This proves (i).

For (ii), the family $\binom{[\ell+1]}\ell$ is contained in
$\overline N_\ell(\mathcal A)$, so the latter is nonempty and
nontrivial. The degree of $1$ in $\mathcal B(n,k,\ell)$ is $S$.
Since the restricted universe is shifted, its degree at any other
point is at most $S$: replacing that point by $1$ gives an injection
from the members containing it but not $1$ into the members
containing $1$ but not that point. A trivial subfamily consequently
has size at most $S$, proving that $\mathcal A$ is nontrivial.
\end{proof}

For a nonempty shifted family
$\mathcal A\subseteq\mathcal B(n,k,\ell)$,
put
\[
\mathcal A_i=\{A\in\mathcal A:\min A=i\},\qquad
\mathcal H_i=\{A\setminus\{i\}:A\in\mathcal A_i\}.
\]
Write $r=\max\{i:\mathcal A_i\neq\emptyset\}$. Then $r\leq\ell$ and
$\mathcal H_i\subseteq\binom{[i+1,n]}{k-1}$.
\begin{lemma}[Disjoint layer decomposition]\label{lem:layers}
With this notation,
\[
N_\ell(\mathcal A)=
\mathop{\dot\bigcup}_{i=1}^{r}
\{[i-1]\cup E:E\in
 N_{\ell-i+1}^{[i+1,n]}(\mathcal H_i)\}.
\]
In particular, the cardinality of $N_\ell(\mathcal A)$ is the sum
of the cardinalities of these $r$ layers.
\end{lemma}
\begin{proof}
If $B\in N_\ell(\mathcal A)$, choose $A\in\mathcal A$ disjoint
from $B$, and let $i$ be the least element of $[r]\setminus B$.
Such an element exists because $\min A\in[r]\setminus B$.
Writing $j=\min A$, we have $i\leq j$. If $i<j$, replace $j$ by
$i$ in $A$; shiftedness ensures that the resulting member still
belongs to $\mathcal A$, and it is still disjoint from $B$.
Thus $B$ is in the displayed $i$-th layer. The reverse inclusion
is immediate. Every member of layer $i$ contains $[i-1]$ and
avoids $i$, so distinct layers are disjoint.
\end{proof}
For $N\geq a+b$ and $0\leq m\leq\binom Na$, write
$\Psi(N,a,b;m)=|N_b^{[N]}(\mathcal L([N],a,m))|$.
Theorem~\ref{cokk} evaluates this quantity by the $(N-a)$-binomial
expansion of $m$, with $\Psi(N,a,b;0)=0$. Lemma~\ref{lem:layers}
therefore yields the lower bound
\begin{equation}\label{eq:layer-lower-bound}
|N_\ell(\mathcal A)|\geq
\sum_{i=1}^{r}\Psi(n-i,k-1,\ell-i+1;|\mathcal A_i|).
\end{equation}

Put
$c_i=\binom{n-i}{k-1}-\binom{n-\ell-1}{k-1}$ for $1\leq i\leq\ell$.
This is the capacity of minimum-element layer $i$ in the restricted
universe. The comparison family has an explicit neighborhood size.
\begin{proposition}\label{prop:restricted-lex-count}
Let $n>k+\ell$, $k,\ell\geq2$, and
$0<m\leq|\mathcal B(n,k,\ell)|$.
There are unique integers $r,t$ with $1\leq r\leq\ell$,
$1\leq t\leq c_r$, and
$m=\sum_{i<r}c_i+t$. Then
\begin{align}
|N_\ell(\mathcal L_\ell([n],k,m))|
 &=\sum_{i<r}\left(\binom{n-i}{\ell-i+1}-1\right)
    +\Psi(n-r,k-1,\ell-r+1;t),\label{eq:restricted-lex-count}\\
|\overline N_\ell(\mathcal L_\ell([n],k,m))|
 &=\binom{n-r+1}{\ell-r+1}+r-1
    -\Psi(n-r,k-1,\ell-r+1;t).\label{eq:restricted-companion-count}
\end{align}
\end{proposition}
\begin{proof}
The restricted lexicographic segment is shifted, since a left shift
stays in the restricted universe and precedes the original set in
lexicographic order. Its first $r-1$ layers are full, its $r$-th
layer has size $t$, and all later layers are empty.

In a full layer $i$, the residual family consists of all
$(k-1)$-sets of $[i+1,n]$ meeting $[i+1,\ell+1]$.
Its disjointness neighborhood on level $\ell-i+1$ consists of
all sets on that level except $[i+1,\ell+1]$. Indeed, every
other such set leaves a point of $[i+1,\ell+1]$ available, and
its complement in $[i+1,n]$ has size $n-\ell-1\geq k-1$.
It therefore avoids some residual member.

All residual sets meeting $[r+1,\ell+1]$ precede those avoiding
that interval. The partial residual layer is consequently an
ordinary lexicographic segment of size $t$, with neighborhood
size $\Psi(n-r,k-1,\ell-r+1;t)$.
Lemma~\ref{lem:layers} now proves
\eqref{eq:restricted-lex-count}. Subtracting from $\binom n\ell$
and using Pascal's identity gives
\eqref{eq:restricted-companion-count}.
\end{proof}

\subsection{Restricted-lex comparison}
\begin{theorem}[Restricted-lex comparison]\label{sec123:cck}
Let $n>k+\ell$, $k,\ell\geq2$, and
$\mathcal A\subseteq\mathcal B(n,k,\ell)$. Then
\[
|N_\ell(\mathcal A)|\geq
|N_\ell(\mathcal L_\ell([n],k,|\mathcal A|))|.
\]
In particular, this holds when $\mathcal A\subseteq\binom{[n]}k$
is shifted and $\overline N_\ell(\mathcal A)$ is nonempty and
nontrivial. For $k,\ell\geq3$, equality at $|\mathcal A|=S+1$
is characterized in Proposition~\ref{prop:threshold-equality} below.
\end{theorem}

\begin{proof}
The empty family is immediate. Fix $m=|\mathcal A|>0$ and
induct lexicographically on
\[
(r,q),\qquad r=\max\{\min A:A\in\mathcal A\},\qquad
q=|\{A\in\mathcal A:\min A=r\}|.
\]
Here $1\leq r\leq\ell$ and $q>0$. For any family
$\mathcal H\subseteq\mathcal B(n,k,\ell)$ whose members meet
$[r]$, and $P,Q\subseteq[r]$, write
\[
\mathcal H_P=\{H\setminus P:H\in\mathcal H,\ H\cap[r]=P\},
\qquad
\mathcal U_Q(\mathcal H)=
\bigcup_{P\subseteq[r]\setminus Q}
N_{\ell-|Q|}^{[r+1,n]}(\mathcal H_P).
\]
Thus $\mathcal H_\emptyset=\emptyset$, all sections have the
common ground set $[r+1,n]$, and
\begin{equation}\label{eq:fixed-prefix-decomposition}
N_\ell(\mathcal H)=
\mathop{\dot\bigcup}_{Q\subseteq[r]}
\{Q\cup E:E\in\mathcal U_Q(\mathcal H)\}.
\end{equation}
Indeed, $Q$ records the intersection of a neighborhood member
with $[r]$. This identity does not require shiftedness.

Replace each nonempty section $\mathcal H_P$ by the ordinary
lexicographic segment of the same size, and denote the resulting
family by $C_r(\mathcal H)$. This stays in
$\mathcal B(n,k,\ell)$: sections with $|P|\geq2$ are
unrestricted, while the allowed residual members of a singleton
section are precisely the initial lexicographic segment meeting
$[r+1,\ell+1]$. By Theorem~\ref{cokk}, each residual neighborhood
does not increase in size. After compression, these neighborhoods
are final lexicographic segments on their common target level,
and hence are nested. To see this, complement and reverse the
labels: an initial lexicographic segment becomes initial colex,
whose shadow is again initial colex. Thus, for $b=\ell-|Q|$,
\[
|\mathcal U_Q(C_r(\mathcal H))|
\leq\max_{P\subseteq[r]\setminus Q}
|N_b^{[r+1,n]}(\mathcal H_P)|
\leq|\mathcal U_Q(\mathcal H)|.
\]
A section of degree $d$ with $d+b>n-r$ has empty neighborhood,
and a nonempty section of degree zero has full neighborhood;
these cases obey the same comparison. Summing
\eqref{eq:fixed-prefix-decomposition} gives
$|N_\ell(C_r(\mathcal H))|\leq|N_\ell(\mathcal H)|$.
If $r=1$, this compression produces
$\mathcal L_\ell([n],k,m)$, proving the base case.

Apply elementary shifts and $C_r$ whenever they change the
family. A shift preserves the restricted universe and does not
increase the neighborhood: shift the family and its full
cross-intersecting companion simultaneously, preserving
cross-intersection and the companion's size. A left shift cannot
increase $r$, or $q$ when $r$ stays fixed, while $C_r$ preserves
$(r,q)$. If this parameter decreases, induction proves the
result immediately. Otherwise each operation strictly decreases
the sum of the global lexicographic ranks of the members.
We therefore reach a shifted family $\mathcal F$ whose sections
are all lexicographic, with the same size and parameter and
$|N_\ell(\mathcal F)|\leq|N_\ell(\mathcal A)|$.

Shifting between prefix points gives
\[
c_r\geq|\mathcal F_{\{1\}}|\geq\cdots
\geq|\mathcal F_{\{r\}}|=q>0.
\]
If every singleton section before $r$ has size $c_r$, each
contains the full residual star at $r+1$, whose first and
further shadows are full. For any $P\subseteq[r]$ with
$2\leq|P|\leq k$, choose $j\in P\setminus\{r\}$.
Successive shifts of residual points to $P\setminus\{j\}$ give
$\partial^{|P|-1}\mathcal F_{\{j\}}\subseteq\mathcal F_P$,
so every such section is full. Consequently all minimum-element
layers before $r$ are full and the last is lexicographic.
Thus $\mathcal F=\mathcal L_\ell([n],k,m)$, as required.

Otherwise choose the least $j<r$ such that
$p=|\mathcal F_{\{j\}}|<c_r$. Put
$z=\binom{n-r-1}{k-2}$, and set $h=z$ if $p<z$, and
$h=c_r$ otherwise. Then $0<q\leq p<h\leq c_r$.
We now transfer members from section $\{r\}$ to section
$\{j\}$ by arranging the binomial expansions in a matrix.
Put $f_i(m)=\Psi(n-i,k-1,\ell-i+1;m)$ for $1\leq i\leq\ell$.
For each nonempty minimum-element layer $\mathcal F_i$, write
\begin{equation}\label{eq:cck-matrix-row}
|\mathcal F_i|
=\sum_{j=0}^{s_i}\binom{n-a_{i,j}}{k-a_{i,j}+i+j-1},
\qquad
i\leq a_{i,0}<\cdots<a_{i,s_i},\quad i+s_i\leq n-k.
\end{equation}
Every displayed summand is positive. Place $a_{i,j}$ in column
$i+j$ of row $i$. The layer sequences form the upper triangular matrix
\[
M=
\begin{pmatrix}
a_{1,0}&a_{1,1}&a_{1,2}&a_{1,3}&\cdots\\
0&a_{2,0}&a_{2,1}&a_{2,2}&\cdots\\
0&0&a_{3,0}&a_{3,1}&\cdots\\
0&0&0&a_{4,0}&\cdots\\
\vdots&\vdots&\vdots&\vdots&\ddots
\end{pmatrix}.
\]
All missing entries are $0$; append zero rows to obtain an
$(n-k)\times(n-k)$ matrix. A zero is an empty-entry marker
and contributes zero to both sums. A nonzero entry $a$ in column
$u$ contributes
\begin{equation}\label{eq:cck-column-contribution}
V_u(a)=
\left(\binom{n-a}{k-a+u-1},\binom{n-a}{\ell-u+1}\right).
\end{equation}
By Theorem~\ref{cokk}, the two sums in row $i$ are
$|\mathcal F_i|$ and $f_i(|\mathcal F_i|)$.
They depend on the entries and their columns, not on their rows.
When sorting a column, place its nonzero entries in nondecreasing
order, followed by zeros.

Exchanging entries within a column preserves both total
contributions. For a positive canonical term in column $u\leq n-k$,
the first contribution is $\binom{n-a}{n-k-u+1}$, which strictly
decreases with $a$. Hence placing a smaller positive entry above
a larger one, or moving a positive term into an empty upper cell,
increases the upper size. For $a<n$, Pascal's identity in both
coordinates gives
\begin{equation}\label{eq:cck-matrix-pascal}
V_u(a)=V_u(a+1)+V_{u+1}(a+1),
\end{equation}
and, for $1\leq d\leq n-a$,
\begin{equation}\label{eq:cck-matrix-refinement}
V_u(a)=\sum_{v=1}^{d}V_{u+v-1}(a+v)+V_{u+d}(a+d).
\end{equation}
If refinement produces a zero-size term, its second contribution
is retained until deletion; append a zero row and column if needed.

Write the canonical rows for $p,q$ as
$(\alpha_0,\ldots,\alpha_s)$ and $(\beta_0,\ldots,\beta_t)$.
For these fixed-prefix sections both expansions start in column
$r$. Place them in rows $j$ and $r$, respectively, with zeros
before column $r$; this is again an upper triangular arrangement.
Only the following two transformation rules are needed.

Suppose first that $i>s$ or $\alpha_i>\beta_i$ for some
$i\leq t$, and take the least such $i$. Since $p\geq q$,
we have $i\geq1$. Exchange the two tails starting in column $r+i$,
moving the lower tail into the upper row and conversely.
An absent tail contributes zero.
Both joins are strict: $\alpha_{i-1}\leq\beta_{i-1}<\beta_i$,
and $\beta_{i-1}<\alpha_i$ if the exchanged upper tail exists.
Both total contributions are unchanged.
If the old upper tail is nonempty, its size is at most
$\binom{n-\alpha_i+1}{n-k-r-i+1}-1$, whereas the incoming
tail has size at least $\binom{n-\beta_i}{n-k-r-i+1}$.
Hence
\[
p'-p=q-q'\geq
\binom{n-\beta_i}{n-k-r-i+1}
-\binom{n-\alpha_i+1}{n-k-r-i+1}+1\geq1.
\]
The bound on the old tail is the defining greedy property
of its canonical binomial expansion. If that tail is empty,
the increase is exactly the positive size of the incoming tail.

Otherwise $s\geq t$ and $\alpha_i\leq\beta_i$ for $0\leq i\leq t$.
Choose the least $d\geq1$ such that $t+d>s$ or
$\beta_t+d<\alpha_{t+d}$. Then
$d\leq s-t+1\leq n-k-r-t+1\leq n-\beta_t$.
Apply \eqref{eq:cck-matrix-refinement} to the last lower
entry $\beta_t$. The entries $\beta_t+1,\ldots,\beta_t+d$
occupy columns $r+t,\ldots,r+t+d-1$, and the final entry
$\beta_t+d$ occupies column $r+t+d$.
Exchange this final entry with the upper tail starting in
that column. The upper row then ends at the inserted entry,
while the lower row retains the refined prefix and receives
the exchanged tail. Both operations preserve both totals
before normalization.
Minimality of $d$ gives
$\alpha_{t+d-1}<\beta_t+d$, and the other join is strict
by the choice of $d$. The same first-entry comparison as above
shows that the upper size strictly increases.

A zero-size term occurs only if $\beta_t=k+r+t-1$.
In that case $\alpha_{t+d},\ldots,\alpha_s$ is empty:
otherwise its first entry would exceed the maximum admissible
label in its column. The zero-size terms are therefore terminal.
Deleting them preserves the total size and cannot increase the
neighborhood sum. A terminal degree-zero term in column
$n-k+1$ has size one and neighborhood contribution zero;
Pascal carrying preserves both totals since $n>k+\ell$.
Thus the normalized rows satisfy $p'>p$, $q'<q$,
$p'+q'=p+q$, and
$f_r(p')+f_r(q')\leq f_r(p)+f_r(q)$.

We verify $p'\leq h$. The canonical row for $h$ is $(r+1)$
when $h=z$, and $(r+1,r+2,\ldots,\ell+1)$ when $h=c_r$.
Both old sizes are smaller than $h$.
Their first labels are therefore at least $r+1$.
Before normalization, every label in column $r+i$ is at least
$r+i+1$, and the exchange and refinement preserve these lower
bounds. A row exceeding $h$ must consequently contain its
entire consecutive prefix and a further term.
After a tail exchange, such a prefix would force one of the
original rows already to contain it: a minimal label at the
cut forces all preceding donor labels to be minimal as well.
After refinement, the upper row ends at the inserted term,
so exceeding $h$ would require the old upper row already to
contain the whole prefix. Neither is possible.
A degree-zero carry increases $p$ by exactly one and likewise
cannot cross the integer bound $h$.
We have therefore obtained
\begin{equation}\label{eq:fixed-prefix-transfer-sizes}
p<p'\leq h,\qquad 0\leq q'<q,\qquad p'+q'=p+q.
\end{equation}

Replace the two sections by lexicographic segments of sizes
$p',q'$ and leave all other sections unchanged, obtaining
$\mathcal F'\subseteq\mathcal B(n,k,\ell)$ with
$|\mathcal F'|=|\mathcal F|$. We next identify the numerical
change with the actual neighborhood change. For distinct
$a,u\in[r]$, shiftedness of the original family gives
\begin{equation}\label{eq:singleton-pair-shadow}
\partial\mathcal F_{\{a\}}\subseteq\mathcal F_{\{a,u\}},
\end{equation}
by replacing a deleted residual point by $u$.
The new donor section is contained in the old one, so the
unchanged pair section $\mathcal F_{\{r,u\}}$ contains a subset
of every old and new donor member.

For the receiving section, if $p\geq z$, the old section contains
the full residual star at $r+1$. Its shadow is
$\binom{[r+1,n]}{k-2}$, so every
$\mathcal F_{\{j,u\}}$, $u\ne j$, is full.
If $p<z$, then $p'\leq z$, so every new receiving member
contains $r+1$. An old member $H$ gives
$H\setminus\{r+1\}\in\mathcal F_{\{j,u\}}$.
This set avoids $r+1$, so the lexicographic pair section contains
every $(k-2)$-set containing $r+1$.
Here $k\geq3$, since $p>0$ and $z=1$ when $k=2$.
Every new receiving member thus contains a member of the pair
section. Together with \eqref{eq:singleton-pair-shadow}, this gives
\begin{equation}\label{eq:preserved-pair-comparison}
N_v^{[r+1,n]}(\mathcal F_{\{a\}})
 \cup N_v^{[r+1,n]}(\mathcal F'_{\{a\}})
 \subseteq N_v^{[r+1,n]}(\mathcal F_{\{a,u\}})
\end{equation}
for $a\in\{j,r\}$, $u\in[r]\setminus\{a\}$ and every
residual level $v$.

In \eqref{eq:fixed-prefix-decomposition}, a block with
$|[r]\setminus Q|\geq2$ is unchanged, since every eligible
changed singleton is dominated by an eligible unchanged pair
section, both before and after the transfer.
The other blocks are unchanged except those with
$[r]\setminus Q=\{j\}$ or $\{r\}$.
Their residual level is $\ell-r+1$, so
\begin{equation}\label{eq:fixed-prefix-transfer-cost}
|N_\ell(\mathcal F')|-|N_\ell(\mathcal F)|
=f_r(p')+f_r(q')-f_r(p)-f_r(q)\leq0.
\end{equation}
Only the original family was required to be shifted.

If $q'>0$, the new parameter is $(r,q')$; if $q'=0$, the
last nonempty layer has index less than $r$. In either case
induction applies to $\mathcal F'$ and gives
\[
|N_\ell(\mathcal A)|\geq|N_\ell(\mathcal F)|
\geq|N_\ell(\mathcal F')|
\geq|N_\ell(\mathcal L_\ell([n],k,m))|.
\]
Finally, the application to shifted families with nontrivial
companion follows from Lemma~\ref{lem:restricted-universe}(i).
\end{proof}

The proof also records exactly when a transfer preserves the
neighborhood size. A column exchange or Pascal refinement
preserves both numerical sums before normalization.
In the refinement case the neighborhood sum decreases strictly
if and only if
\[
\beta_t=k+r+t-1\qquad\text{and}\qquad r+t\leq\ell+1.
\]
Indeed, under the first condition the first deleted term has
second contribution $\binom{n-k-r-t}{\ell-r-t+1}$, which is
positive precisely under the second condition; if the second
condition fails, all deleted second contributions vanish.
The condition $\beta_t=k+r+t-1$ says that the last binomial
term of the donor row has size one. In particular $q=1$
gives a strict decrease, but the entire donor row need not
have size one for a strict decrease to occur.

For $k=\ell=3$, the matrix gives a useful equality-preserving transfer.
The two nonempty rows for the triangle family are $(2,3)$ and
$(3)$, starting in columns $1$ and $2$. Apply
$V_2(3)=V_2(4)+V_3(4)$ to the lower row and move its new
entry in column $3$ into the empty upper cell.
The rows become $(2,3,4)$ and $(4)$, representing
$\mathcal L_3([n],3,3n-8)$.
No zero-size term is deleted, so both sums are unchanged.
Thus $(2,3)\to(2,3,4)$ is an equality-preserving transfer.

\subsection{Equality at the first nontrivial threshold}

The next proposition concerns $k,\ell\geq3$, the range needed for
Theorem~\ref{mthm}, and describes the equality cases up to isomorphism.
\begin{proposition}[Equality at the threshold]\label{prop:threshold-equality}
Let $n>k+\ell$ and $k,\ell\geq3$, and put
\[
 S=\binom{n-1}{k-1}-\binom{n-\ell-1}{k-1},\qquad
 T=\binom{n-1}{\ell-1}-\binom{n-k-1}{\ell-1}.
\]
Suppose that $\mathcal A\subseteq\binom{[n]}k$ and
$\mathcal B=\overline N_\ell(\mathcal A)$ are nonempty and
nontrivial, and that $|\mathcal A|=S+1$. Then
\[
 |N_\ell(\mathcal A)|\geq\binom n\ell-T-1
 =|N_\ell(\mathcal L_\ell([n],k,S+1))|.
\]
Equality holds if and only if, up to a common relabeling of $[n]$,
the pair has one of the following forms:
\begin{enumerate}
\item[\rm(i)] There are $X\in\binom{[2,n]}k$ and
$Y\in\binom{[2,n]}\ell$ with
$X\cap Y\neq\emptyset$ such that
\begin{align*}
 \mathcal A&=\{A\in\tbinom{[n]}k:1\in A,\ A\cap Y\neq\emptyset\}
                   \cup\{X\},\\
 \mathcal B&=\{B\in\tbinom{[n]}\ell:1\in B,\ B\cap X\neq\emptyset\}
                   \cup\{Y\}.
\end{align*}
\item[\rm(ii)] $k=\ell=3$, and there are
$P,Q\in\binom{[2,n]}2$ with
$P\cap Q\neq\emptyset$ such that
\begin{align*}
 \mathcal A&=\{A\in\tbinom{[n]}3:1\in A,\ A\cap Q\neq\emptyset\}
       \cup\{A\in\tbinom{[n]}3:1\notin A,\ P\subseteq A\},\\
 \mathcal B&=\{B\in\tbinom{[n]}3:1\in B,\ B\cap P\neq\emptyset\}
       \cup\{B\in\tbinom{[n]}3:1\notin B,\ Q\subseteq B\}.
\end{align*}
\end{enumerate}
In particular, this characterizes equality in
Theorem~\ref{sec123:cck} at $|\mathcal A|=S+1$ for
$k,\ell\geq3$.
\end{proposition}
\begin{proof}
We first record the relevant special case of M\"ors' equality
criterion. For $\mathcal H\subseteq\binom{[2,n]}{k-1}$ of
size $S$, the canonical $(n-k)$-binomial expansion is
\[
S=\sum_{i=2}^{\ell+1}\binom{n-i}{n-k-i+2}.
\]
It has exactly $\ell$ terms. Equivalently, its last lower index
is $n-k-\ell+1>n-k-\ell$, so Theorem~\ref{cokk} gives
$|N_\ell^{[2,n]}(\mathcal H)|\geq\binom{n-1}{\ell}-1$,
with equality if and only if, for some $Y\in\binom{[2,n]}\ell$,
\begin{equation}\label{eq:threshold-star-equality}
\mathcal H=\{H\in\tbinom{[2,n]}{k-1}:H\cap Y\neq\emptyset\}.
\end{equation}
Indeed, this is the ordinary lexicographic segment of size $S$
when $Y=[2,\ell+1]$, and its neighborhood consists of all
$\ell$-sets of $[2,n]$ other than $Y$. This determines the
star part up to isomorphism without a shiftedness assumption.

We use F\"uredi's threshold theorem and its equality analysis
\cite[Section~3]{Furedi1995}:
if a nontrivial cross-intersecting pair has sizes greater than
$S$ and $T$, respectively, then its sizes are exactly $S+1$
and $T+1$. For $k,\ell\geq3$ with $k+\ell\geq7$, its
equality classification gives precisely (i).
Consequently $|\mathcal B|\leq T+1$, and, outside
$k=\ell=3$, equality has the stated form. When $k=\ell=3$,
equality is equivalent to
$|\mathcal A|=|\mathcal B|=3n-8$; the classification in
\cite[Proposition~7.2]{Huang2026} gives (i) and (ii).
In both cases we label the distinguished point $1$.
In (i), the residual star part has size $S$ and every member
meets the exceptional set $Y$ of $\mathcal B$. Its neighborhood
therefore omits $Y$, so equality holds in the bound above and
its form is given by \eqref{eq:threshold-star-equality}.
The other star part is determined in the same way.

Conversely, the two displayed constructions are cross-intersecting
and nontrivial. In (i) their sizes are $S+1$ and $T+1$.
In (ii) each family has size $(2n-5)+(n-3)=3n-8$.
In both cases the displayed $\mathcal B$ is the full companion.
Indeed, in (i) an $\ell$-set avoiding $1$ and meeting every
member of the star part of $\mathcal A$ must contain $Y$,
and hence equal $Y$; a set containing $1$ must meet $X$.
In (ii) a triple avoiding $1$ must contain $Q$, whereas a
triple containing $1$ must meet $P$.
These assertions follow by extending a missed prescribed point
to a disjoint member, using $n>k+\ell$.

The restricted-lex family of size $S+1$ is the family in (i)
with $X=[2,k+1]$ and $Y=[2,\ell+1]$.
Its companion has size $T+1$, proving the displayed identity.
Finally, if $\mathcal A\subseteq\mathcal B(n,k,\ell)$ and
$|\mathcal A|=S+1$, Lemma~\ref{lem:restricted-universe}(ii)
supplies both nontriviality assumptions.
\end{proof}

\section{The product of critically shifted families}\label{sec:critical-revised}
Throughout this section, $k,\ell\geq3$,
$n\geq\max\{2k,2\ell\}$, and $n>k+\ell$, unless stated otherwise.
We distinguish critical shiftedness of a pair from the following
single-family condition: a nontrivial family $\mathcal F$ is
\emph{critically shifted} if it is not shifted and, for every $i<j$,
either $s_{ij}(\mathcal F)=\mathcal F$ or $s_{ij}(\mathcal F)$ is
trivial. Put
\begin{align*}
Q(n,k,\ell)&=\sum_{i=1}^{2}
 \left(\binom{n-i}{\ell-1}-\binom{n-k-i}{\ell-1}\right),\\
F(n,k,\ell)&=\left(\binom{n-2}{k-2}+1\right)Q(n,k,\ell).
\end{align*}
For disjoint sets $R\in\binom{[n]}2$ and $X\in\binom{[n]}k$, define
\begin{align*}
\mathcal C_k(R,X)&=\{A\in\tbinom{[n]}k:R\subseteq A\}\cup\{X\},\\
\mathcal D_\ell(R,X)&=\{B\in\tbinom{[n]}\ell:
 B\cap R\neq\emptyset,\ B\cap X\neq\emptyset\}.
\end{align*}
We first establish the two-core decomposition of critically shifted
families. Theorem~\ref{thm:two-core-product} gives the sharp product
bound for every cross-intersecting pair admitting this decomposition,
without assuming nontriviality. Corollary~\ref{critically} then
applies it to the critically shifted case.

\subsection{Critical pairs and section cores}
The structural statements in Lemmas~\ref{lem:critical-sections},
\ref{lem:section-cores} and~\ref{lem:two-cores} do not require
the standing numerical assumptions of this section.
They use only their explicitly stated hypotheses and, in
Lemma~\ref{lem:two-cores}, the setup of
Lemma~\ref{lem:section-cores}.

\begin{lemma}\label{lem:critical-sections}
Let $\mathcal F\subseteq\binom{[n]}a$ be nontrivial.
The shift along an ordered pair $(i,j)$ makes $\mathcal F$ trivial
if and only if
\[
\mathcal F(\overline{ij})=\emptyset,\qquad
\mathcal F(i,\overline j)\cap\mathcal F(j,\overline i)=\emptyset.
\]
In that event both single-point sections are nonempty and the common
intersection after the shift is exactly $\{i\}$.
\end{lemma}
\begin{proof}
The receiving point $i$ is the only point whose membership can be
newly acquired by a set. Thus it is the only possible new common
point. A member avoiding both points would remain and omit $i$.
If the two sections had a common member $E$, the set
$E\cup\{j\}$ would also remain, since its $i$-counterpart already
belongs to the family. Conversely, the two displayed conditions
ensure that every member omitting $i$ is shifted to one containing
$i$. Nontriviality supplies members omitting each of the two points,
so both sections are nonempty.
\end{proof}
We call the unordered pair $\{i,j\}$ \emph{critical} for
$\mathcal F$ when the conditions of this lemma hold. Thus criticality
is independent of the orientation. A pair avoided by a member of a
critically shifted family is not critical, and its left shift fixes
the family. We shall repeatedly use this observation.

\begin{lemma}[Separation of the two families]\label{lem:critical-separation}
Let $a,b\geq3$ and $n\geq2\max\{a,b\}$. Suppose that
$\mathcal F\subseteq\binom{[n]}a$ and
$\mathcal G\subseteq\binom{[n]}b$ are saturated, nontrivial,
cross-intersecting families, and that every simultaneous left shift
either fixes both families or makes at least one trivial.
At most one of the two families has a critical pair.
\end{lemma}
\begin{proof}
Suppose $e$ is critical for $\mathcal F$ and $f$ is critical for
$\mathcal G$. Saturation implies that $\mathcal F$ contains every
$a$-set containing $f$, and $\mathcal G$ contains every $b$-set
containing $e$. The pairs cannot be disjoint, since an $a$-set
containing $f$ and avoiding $e$ could then be chosen.

If $e\neq f$, write $e=\{x,y\}$ and $f=\{x,z\}$. Every member of
$\mathcal F$ omitting $x$ contains $y$ and omits $z$: otherwise its
replacement of $y$ by $x$ lies in the full $f$-star, contradicting
the disjoint sections at $e$. Similarly, every member of
$\mathcal G$ omitting $x$ contains $z$ and omits $y$.
Choose $F_0\in\mathcal F$ omitting $x$.
The set $(F_0\setminus\{y\})\cup\{x\}$ meets every member of
$\mathcal G$: those containing $x$ meet it at $x$, and those
omitting $x$ also omit $y$, so their original intersection with
$F_0$ is retained. Saturation puts this replacement in
$\mathcal F$, again contradicting criticality at $e$.

It follows that every critical pair of either family is the same
pair $e$. All shifts within $[n]\setminus e$ therefore fix both
families. The two nonempty single-point sections of $\mathcal F$
at $e$ are shifted on that ground set, so they both contain its
first $(a-1)$-set. This contradicts their disjointness.
\end{proof}
Consequently, in a saturated critically shifted pair exactly one
family admits a trivializing shift, and that family is critically
shifted in the single-family sense.

\begin{lemma}\label{lem:section-cores}
Let $\mathcal A\subseteq\binom{[n]}k$ be nontrivial and critically
shifted, with $k\geq3$. There is a critical pair $\{1,j\}$.
Choose the least such $j$, set $U=[n]\setminus\{1,j\}$, and put
\[
C=\bigcap_{E\in\mathcal A(1,\overline j)}E,\qquad
D=\bigcap_{E\in\mathcal A(j,\overline1)}E.
\]
Then $C,D\neq\emptyset$, $[2,j-1]\subseteq C$, and
$|C\cup D|\geq k$. Moreover:
\begin{enumerate}
\item[\rm(i)] $E\setminus\{x\}\in\mathcal A(1j)$ whenever
$E\in\mathcal A(1,\overline j)$ and $x\in E\setminus C$;
the analogous assertion holds for the other section and $D$;
\item[\rm(ii)] $\mathcal A$ is fixed under all left shifts within $U\setminus C$
and within $U\setminus D$;
\item[\rm(iii)] every member of $\mathcal A(1j)$ contains $C\setminus D$
or $D\setminus C$.
\end{enumerate}
\end{lemma}
\begin{proof}
Suppose no critical pair contains $1$, and choose a critical pair
$\{u,v\}$. All shifts involving $1$ fix $\mathcal A$. A member
omitting $1$ must contain both $u,v$, since otherwise shifting its
sole point in that pair to $1$ would produce a member avoiding the
pair. Such a member exists by nontriviality. Replacing $u$ and $v$,
respectively, by $1$ produces a common member of the two
single-point sections at $\{u,v\}$, a contradiction.

For the least $j$ in the statement, the two sections are nonempty
and disjoint. If $2\leq x<j$ and $x\notin C$, a member of the
$1$-section omits $x$, so the pair $\{x,j\}$ is not critical and
its shift fixes $\mathcal A$. Every member of the $j$-section
therefore contains $x$; otherwise shifting $j$ to $x$ would produce
a member avoiding $\{1,j\}$. Minimality of $j$ also fixes the
$(1,x)$-shift. Applying these fixed shifts successively to
$\{j\}\cup E$ gives $\{1,j\}\cup(E\setminus\{x\})$ and then
$\{1\}\cup E$. This contradicts disjointness of the sections.
Thus $[2,j-1]\subseteq C$.

If $x\in U\setminus C$, then $x>j$. A member of the $1$-section
omitting $x$ shows that $\{j,x\}$ is not critical; its fixed shift
proves (i). For the other section use the fixed $(1,x)$-shift when
$x\notin D$. For $u<v$ in $U\setminus C$, choose a member of the
$1$-section omitting $u$. If it also omits $v$, it gives a member
avoiding the pair. Otherwise (i) gives such a member by deleting
$v$ and adjoining $j$. This proves (ii), and the argument for $D$
is identical. If either core were empty, both sections would be
shifted on all of $U$ and contain the same first $(k-1)$-set.
Hence both cores are nonempty.

The intersection of a nonempty family fixed under left shifts is
an initial segment of its ordered ground set: a compulsory point
cannot have an earlier noncompulsory point. Apply this to the
projections of the $j$-section onto $U\setminus C$, and of the
$1$-section onto $U\setminus D$. Thus $D\setminus C$ and
$C\setminus D$ are initial segments of these respective ground sets.
Suppose that $|C\cup D|\leq k-1$, and write $V=C\cup D$.
Let $W$ consist of the first $k-1-|V|$ points of $U\setminus V$.
Since the $1$-section is fixed under all left shifts within
$U\setminus C$, it contains $C$ together with the first
$k-1-|C|$ points of $U\setminus C$.
As $D\setminus C$ is an initial segment of that ground set,
this member is exactly $V\cup W$.
The same argument with $C$ and $D$ interchanged shows that
$V\cup W$ also belongs to the $j$-section, contradicting
the disjointness of the two sections.

Finally, if $Z\in\mathcal A(1j)$ omits
$c\in C\setminus D$ and $d\in D\setminus C$, the member
$\{1,j\}\cup Z$ avoids $\{c,d\}$, so its shift is fixed.
If $c<d$, apply it to a member of the $j$-section omitting $c$;
this removes the compulsory point $d$. If $d<c$, use the other
section. Both cases are contradictions, proving (iii).
\end{proof}

\subsection{The section partition and the two-core reduction}
\label{subsec:core-revised}
Continue with the least critical pair $\{1,j\}$ from
Lemma~\ref{lem:section-cores}, and put
$C_1=C\setminus D$ and $D_1=D\setminus C$.
Define
\begin{equation}\label{eq:prescribed-partition}
\begin{split}
\mathcal Z_C
 &=\{Z\in\mathcal A(1j):C_1\subseteq Z,\ D_1\nsubseteq Z\},\\
\mathcal Z_D
 &=\{Z\in\mathcal A(1j):D_1\subseteq Z,\ C_1\nsubseteq Z\},\\
\mathcal Z_0
 &=\{Z\in\mathcal A(1j):C_1\cup D_1\subseteq Z\}.
\end{split}
\end{equation}
By Lemma~\ref{lem:section-cores}(iii), the three families in
\eqref{eq:prescribed-partition} form a disjoint partition of
$\mathcal A(1j)$.

\begin{lemma}[Two disjoint cores]\label{lem:two-cores}
There are disjoint sets $R,S$ such that
$2\leq|R|,|S|\leq k$, $|R|+|S|\geq k+2$, and every member of
$\mathcal A$ contains $R$ or $S$. The two resulting parts are
nonempty and disjoint.
\end{lemma}
\begin{proof}
Put $I=C\cap D$ and $V=C\cup D$. Both $C_1,D_1$ are nonempty,
because $|V|\geq k$ and $|C|,|D|\leq k-1$.

If $\mathcal Z_C=\emptyset$, the deletion assertion in
Lemma~\ref{lem:section-cores} forces $|C|=k-1$. Indeed, otherwise
one can delete a point outside $C$ from a member of the $1$-section;
the resulting double-section member still omits a point of $D_1$,
since $|C\cup D|\geq k$. Thus the $1$-section is $\{C\}$.
For any $d\in D_1$, every member except $\{1\}\cup C$ contains
$\{j,d\}$, which is disjoint from that exceptional member.
These give the required cores. The case $\mathcal Z_D=\emptyset$
is symmetric.

Assume now that $\mathcal Z_C$ and $\mathcal Z_D$ are both
nonempty, and set
\[
\widehat C_2=\bigcap_{Z\in\mathcal Z_C}Z,\qquad
D_2=\bigcap_{Z\in\mathcal Z_D}Z.
\]
If $|D|<k-1$, intersect the sets $E\setminus\{x\}$ with
$E\in\mathcal A(j,\overline1)$ and $x\in E\setminus D$.
They all belong to $\mathcal Z_D$, and their intersection is $D$;
hence $D_2\subseteq D$. Similarly,
$\widehat C_2\subseteq C$ if $|C|<k-1$.
If $c\in D_2\cap C_1$, it follows that the $j$-section is the
singleton $\{D\}$. All other members contain $\{1,c\}$, disjoint
from $\{j\}\cup D$, so we are again in the exceptional-member
case. Treat $\widehat C_2\cap D_1\neq\emptyset$ symmetrically.
We may therefore assume
\begin{equation}\label{eq:cross-core-exclusion}
D_2\cap C_1=\widehat C_2\cap D_1=\emptyset.
\end{equation}

We next prove $D_2\subseteq D$ when $|D|=k-1$.
If $x\in D_2\setminus D$, then $x\notin V$ by
\eqref{eq:cross-core-exclusion}. Every $c\in C_1$ precedes $x$:
otherwise the fixed $(x,c)$-shift within $U\setminus D$, applied
to a member of the $1$-section omitting $x$, would remove the
compulsory point $c$.
Choose $Z\in\mathcal Z_D$. It contains $x$ and omits at least one
point of $C_1$. It cannot omit two, since replacing $x$ by either
one using a fixed shift would leave a member of $\mathcal Z_D$
omitting $x$. Thus it omits exactly one, say $c$, and
$Z'=(Z\setminus\{x\})\cup\{c\}\in\mathcal Z_0$.
If $z\in I\setminus Z$, the pair $\{x,z\}$ is avoided by
$\{1,j\}\cup Z'$ and hence has a fixed shift. For $z<x$, its
application to $Z$ gives a member of $\mathcal Z_D$ omitting $x$.
For $x<z$, its application to the unique member $D$ of the
$j$-section changes that section. Both are impossible. Thus
$I\subseteq Z$, and $Z'$ contains $V$, contrary to
$|Z'|=k-2<|V|$. The symmetric argument gives
\begin{equation}\label{eq:aux-core-containment}
C_1\subseteq\widehat C_2\subseteq C,\qquad
D_1\subseteq D_2\subseteq D.
\end{equation}

We claim that $\widehat C_2\cup D_2=V$.
If $x\in I\setminus(\widehat C_2\cup D_2)$, choose
$Z_C\in\mathcal Z_C$ and $Z_D\in\mathcal Z_D$ omitting $x$,
and $d\in D_1\setminus Z_C$.
The pair $\{x,d\}$ is avoided by $\{1,j\}\cup Z_C$, so its
shift is fixed. If $x<d$, applying it to $Z_D$ produces a member
containing neither $C_1$ nor $D_1$, a contradiction. If $d<x$,
apply it to a member of the $1$-section omitting $d$; the result
omits the compulsory point $x$. This proves the claim.

Put $J=\widehat C_2\cap D_2$. First suppose $J=\emptyset$.
Every $Z\in\mathcal Z_0$ contains $\widehat C_2$ or $D_2$.
Indeed, otherwise choose $x\in\widehat C_2\setminus Z$ and
$y\in D_2\setminus Z$. Then $y\in I\setminus\widehat C_2$;
choose $Z_C\in\mathcal Z_C$ omitting $y$ and
$d\in D_1\setminus Z_C$. The pair $\{y,d\}$ is not critical.
If $d<y$, its fixed shift applied to a member of the $1$-section
omitting $d$ removes the compulsory point $y$. If $y<d$, its
application to $Z$ gives a member of $\mathcal Z_C$ still omitting
$x$. Both are contradictions. Therefore
$R=\{1\}\cup\widehat C_2,\  S=\{j\}\cup D_2$
are disjoint cores covering all members, and
$|R|+|S|=|V|+2\geq k+2$.

Finally suppose $J\neq\emptyset$. We show that any
$Z\in\mathcal Z_0$ omitting a point $x\in J$ contains $V\setminus J$.
If $z\in(\widehat C_2\setminus D_2)\setminus Z$, then $z\in I$.
Choose $Z_D\in\mathcal Z_D$ omitting $z$ and
$c\in C_1\setminus Z_D$. The pair $\{z,c\}$ is not critical.
If $c<z$, its fixed shift applied to a member of the $j$-section
omitting $c$ removes the compulsory point $z$.
If $z<c$, its application to $Z$ produces a member of
$\mathcal Z_D$ still omitting $x\in D_2$. Both are impossible.
Interchanging the two roles proves $D_2\setminus\widehat C_2\subseteq Z$
as well. Hence every member contains one of the disjoint sets
$R=J,\  S=\{1,j\}\cup(V\setminus J)$.
Nontriviality supplies a member not containing $J$, so $|S|\leq k$.
Since $|R|+|S|=|V|+2\geq k+2$, both core sizes are at least two.
In either case, no $k$-set contains both cores, and neither part is
empty, since otherwise its core would be common to all members.
\end{proof}

\begin{corollary}\label{cor:critical-size}
If $n\geq2k$ and $k\geq3$, every nontrivial critically shifted
$\mathcal A\subseteq\binom{[n]}k$ satisfies
$|\mathcal A|\leq\binom{n-2}{k-2}+1$.
Unless it has a two-element core and one exceptional member disjoint
from that core, the inequality is strict.
\end{corollary}
\begin{proof}
Shrink the disjoint cores in Lemma~\ref{lem:two-cores}, if necessary,
so that their sizes sum to $k+2$. Write
$g(t)=\binom{n-t}{k-t}$. The differences
$g(t)-g(t+1)=\binom{n-t-1}{k-t}$ are positive and strictly decrease
for $2\leq t<k$. Thus $g(r)+g(k+2-r)$ is maximized at $r=2$ or
$r=k$, with value $\binom{n-2}{k-2}+1$.
If both sizes are at least three the comparison is strict; if one
is two the other is $k$ and its part is a single member.
\end{proof}

\subsection{The density inequality and the one-exception case}

\begin{lemma}\label{lem:density}
Let $m,a,c$ be positive integers with $m>a+c$. If
$\mathcal H\subseteq\binom{[m]}a$ and
$\mathcal G\subseteq\binom{[m]}c$ are cross-intersecting, put
\[
u=\frac{|\mathcal H|}{\binom ma},\qquad
v=\frac{|\mathcal G|}{\binom mc},\qquad
L=\frac{ac}{m(m-a-c)}.
\]
Then $u+v+uv/L\leq1$.
\end{lemma}
\begin{proof}
We include the spectral calculation. Interchanging the families if
necessary, assume $a\leq c$. Equip each uniform layer with the
uniform probability inner product. Define
\[
(Tf)(A)=\binom{m-a}{c}^{-1}
\sum_{\substack{C\in\binom{[m]}c\\C\cap A=\varnothing}}f(C).
\]
Its adjoint averages over the $a$-sets disjoint from a given $c$-set.
For $J\subseteq[m]$, let $f_J^{(b)}(E)=\mathbf1_{J\subseteq E}$
on the $b$-uniform layer. Write $(x)_j=x(x-1)\cdots(x-j+1)$.
For $|J|=j\leq a$, direct counting gives
\[
Tf_J^{(c)}=\frac{(c)_j}{(m-a)_j}
 \sum_{I\subseteq J}(-1)^{|I|}f_I^{(a)},\qquad
T^*f_J^{(a)}=\frac{(a)_j}{(m-c)_j}
 \sum_{I\subseteq J}(-1)^{|I|}f_I^{(c)}.
\]
Let $V_j$ be the span of $f_I^{(a)}$ with $|I|\leq j$.
Then $V_a$ is the entire function space, $TT^*$ preserves $V_j$,
and its action on $V_j/V_{j-1}$ is scalar multiplication by
$\theta_j=\frac{(a)_j(c)_j}{(m-a)_j(m-c)_j}$.
By self-adjointness the same holds on $V_j\cap V_{j-1}^{\perp}$.
The successive factors are less than one, so
$\theta_j\leq\theta_1$ for $j\geq1$. Consequently the norm of
$T$ between the mean-zero subspaces is at most
$\rho=\sqrt{\frac{ac}{(m-a)(m-c)}}$.
For $f=\mathbf1_{\mathcal H}-u$ and
$g=\mathbf1_{\mathcal G}-v$, cross-intersection gives
$\langle f,Tg\rangle=-uv$. Hence
$uv\leq\rho\sqrt{u(1-u)v(1-v)}$.
If $uv>0$, squaring and cancelling $uv$ yields
$uv\leq\rho^2(1-u)(1-v)$. Since $\rho^2=L/(1+L)$,
this is the desired inequality. If $uv=0$, the conclusion follows
from $0\leq u,v\leq1$.
\end{proof}

\begin{lemma}[One exceptional member]\label{asd}
Let $a,b\geq3$, $n\geq2\max\{a,b\}$, and $n>a+b$.
Let $R\in\binom{[n]}2$, $X\in\binom{[n]\setminus R}a$, and
\[
\mathcal F=\{X\}\cup\{R\cup H:H\in\mathcal H\},\qquad
\varnothing\ne\mathcal H\subseteq\binom{[n]\setminus R}{a-2}.
\]
Every family $\mathcal G\subseteq\binom{[n]}b$ cross-intersecting
with $\mathcal F$ satisfies
\[
|\mathcal F||\mathcal G|
\leq\left(\binom{n-2}{a-2}+1\right)Q(n,a,b).
\]
Equality holds precisely when $\mathcal H=\binom{[n]\setminus R}{a-2}$
and $\mathcal G$ consists of all $b$-sets meeting both $R$ and $X$.
\end{lemma}
\begin{proof}
Put $D=\binom{n-2}{a-2}$, $E=\binom{n-2}b$,
$h=|\mathcal H|$, and $Q=Q(n,a,b)$.
The members of $\mathcal G$ meeting $R$ must also meet $X$;
there are at most $Q$ of them. The remaining family
$\mathcal Z=\{G\in\mathcal G:G\cap R=\varnothing\}$ is
cross-intersecting with $\mathcal H$ on $[n]\setminus R$.
Writing $z=|\mathcal Z|$, we have
\begin{equation}\label{eq:one-basic}
|\mathcal F||\mathcal G|\leq(h+1)(Q+z).
\end{equation}
Set $\beta=Q/E$ and
\[
L=\frac{(a-2)b}{(n-2)(n-a-b)},\qquad
\rho^2=\frac{(a-2)b}{(n-a)(n-b-2)}=\frac L{1+L}.
\]
We first prove
\begin{equation}\label{eq:one-strict}
\beta>(1+1/D)\rho^2.
\end{equation}
Let $P$ be the proportion of $(b-1)$-sets of $[n]\setminus R$
meeting $X$. Sampling a uniform two-element subset of such a set
gives
\[
P\geq1-\frac{\binom{n-a-2}2}{\binom{n-2}2}
=\frac{a(2n-a-5)}{(n-2)(n-3)}.
\]
Counting sets containing exactly one point of $R$ gives
$Q\geq2\binom{n-2}{b-1}P$, and therefore
\[
\frac\beta{\rho^2}\geq
\frac{2a(2n-a-5)(n-a)(n-b-2)}
{(a-2)(n-2)(n-3)(n-b-1)}.
\]
Since $a,b\leq n/2$, the four factors satisfy
\[
\frac a{a-2}\geq\frac n{n-4},\quad
2n-a-5\geq\frac{3n-10}2,\quad
n-a\geq\frac n2,\quad
\frac{n-b-2}{n-b-1}\geq\frac{n-4}{n-2}.
\]
It follows that
\[
\frac\beta{\rho^2}\geq
\frac{n^2(3n-10)}{2(n-2)^2(n-3)}>\frac32.
\]
Here $n\geq7$ and $D\geq n-2$, so $1+1/D\leq6/5$,
proving \eqref{eq:one-strict}.

Lemma~\ref{lem:density} applied to $\mathcal H,\mathcal Z$ yields
\begin{equation}\label{eq:one-z}
z\leq\frac{EL(D-h)}{DL+h}.
\end{equation}
We claim that
\begin{equation}\label{eq:one-ratio}
\frac{EL(h+1)}{DL+h}<Q\qquad(1\leq h\leq D).
\end{equation}
If $DL\geq1$, the ratio $(h+1)/(DL+h)$ is nondecreasing in $h$;
its value at $D$ gives
$\frac{EL(h+1)}{DL+h}\leq E(1+1/D)\rho^2<Q$.
If $DL<1$, necessarily $a=3$: otherwise
$2\leq a-2\leq(n-4)/2$, $D\geq\binom{n-2}2$, and
$DL\geq\frac{(n-3)(a-2)b}{2(n-a-b)}>\frac{(a-2)b}2\geq3$.
Thus $a=3$, $D=n-2$, and the ratio in \eqref{eq:one-ratio}
is maximized at $h=1$. Its value is
\[
\frac{2EL}{DL+1}=\frac{2Eb}{(n-2)(n-3)}
=\frac2{n-3}\binom{n-3}{b-1}.
\]
On the other hand,
\[
Q\geq2\left(\binom{n-2}{b-1}-\binom{n-5}{b-1}\right)
\geq6\binom{n-5}{b-2},
\]
and the ratio of the preceding upper bound to this lower bound is
$\frac{n-4}{3(b-1)(n-b-2)}<1$
because $n\geq2b$ and $b\geq3$. This proves
\eqref{eq:one-ratio} in both cases.
If $h<D$, equations \eqref{eq:one-z}--\eqref{eq:one-ratio} give
$(h+1)z<(D-h)Q$, so \eqref{eq:one-basic} is strictly less than
$(D+1)Q$. If $h=D$, cross-intersection forces $z=0$, since
$n>a+b$. Equality then requires all $b$-sets meeting $R$ and $X$
to belong to $\mathcal G$. The converse is immediate.
\end{proof}

\begin{proposition}\label{prop:two-point-construction}
Let $n\geq2\max\{k,\ell\}$, $n>k+\ell$, and $k,\ell\geq3$.
For $R=\{1,2\}$ and $X=[3,k+2]$, the family
$\mathcal C_k(R,X)$ is nontrivial and critically shifted.
Its full cross-intersecting mate is $\mathcal D_\ell(R,X)$, which
is nontrivial, and their product is $F(n,k,\ell)$.
\end{proposition}
\begin{proof}
The full $R$-star has intersection $R$, and $X\cap R=\varnothing$,
so $\mathcal C_k(R,X)$ is nontrivial. The shift $s_{13}$ makes
it trivial. Every shift either fixes the family, or moves
$X$ by replacing one of its points with $1$ or $2$, leaving a
family with a common point. Thus it is critically shifted.
A set meeting both $R$ and $X$ meets every member. Conversely,
an $\ell$-set avoiding $R$ has a disjoint member of the full
$R$-star because $n-\ell-2\geq k-2$, and a set avoiding $X$
is disjoint from the exceptional member. This proves the assertion
about the full mate. The latter has no common point: for any given
point one can choose a different point of $R$, a point of $X$
different from it, and complete these to an $\ell$-set avoiding it.
Finally, separating sets containing $1$ from those containing $2$
but not $1$ gives $|\mathcal D_\ell(R,X)|=Q(n,k,\ell)$.
\end{proof}

\begin{lemma}[Comparison in the other orientation]\label{lem:orientation}
If $n\geq2\ell\geq2k\geq6$ and $n>k+\ell$, then
$F(n,k,\ell)<M_1(n,k,\ell)$.
\end{lemma}
\begin{proof}
Put $J=\binom{n-2}{k-2}$ and
\[
S=\binom{n-1}{k-1}-\binom{n-\ell-1}{k-1},\qquad
T=\binom{n-1}{\ell-1}-\binom{n-k-1}{\ell-1}.
\]
We first show $S\geq2J+2$. Since $\ell\geq k$,
\[
\frac SJ\geq\sum_{j=2}^{k+1}\frac{\binom{n-j}{k-2}}{J}.
\]
Each summand is nondecreasing in $n$: it is the proportion of
$(k-2)$-sets avoiding a fixed $(j-2)$-set in an $(n-2)$-element
set. At $n=2k+1$, the corresponding difference satisfies
\[
\sum_{j=2}^{k+1}\binom{2k+1-j}{k-2}
-2\binom{2k-1}{k-2}
=\frac{2}{k-1}\binom{2k-1}{k-2}-k\geq2.
\]
For $k=3$ the last inequality is equality; for $k\geq4$, use
$\binom{2k-1}{k-2}\geq\binom{2k-1}2\geq(k-1)(k+2)/2$.
Monotonicity of the ratios, together with
$J\geq\binom{2k-1}{k-2}$, proves the claim.
The two summands in $Q(n,k,\ell)$ are nonincreasing in their
index, so $Q(n,k,\ell)\leq2T$. Consequently
\[
F(n,k,\ell)\leq2(J+1)T<(S+1)(T+1)=M_1(n,k,\ell).
\]
\end{proof}

\subsection{The estimate for two nonempty core parts}

We first settle the two smallest uniformities directly.
\begin{proposition}\label{prop:small-critical}
Theorem~\ref{thm:two-core-product}, including its equality characterization,
holds for $k=3$ and $k=4$.
\end{proposition}
\begin{proof}
Take the disjoint cores $R,S$ in the hypotheses, and shrink them
until their sizes sum to $k+2$. No $k$-set contains both reduced
cores, so the original two nonempty parts are retained.
If one part is a singleton, say $\{X\}$ in the $S$-part, then
$|R\setminus X|\geq|R|-(k-|S|)=2$.
Every other member contains any chosen two-element subset of
$R\setminus X$. Lemma~\ref{asd} gives the bound and its equality
characterization. For $k=3$ a core has size three, so this settles
all cases.

For $k=4$, the only remaining case has $|R|=|S|=3$, with both
parts of size at least two. Write
\[
\mathcal A=\{R\cup\{x\}:x\in X\}\cup\{S\cup\{y\}:y\in Y\}.
\]
Put $a=|X|$ and $b=|Y|$, so $2\leq a,b\leq n-3$. Let
\[
W=\binom n\ell-2\binom{n-3}\ell+\binom{n-6}\ell,
\qquad f(t)=\binom{n-3-t}{\ell-t}.
\]
A member of $\mathcal B$ either meets both cores, or avoids $R$
and contains $X$, or avoids $S$ and contains $Y$. Therefore
$|\mathcal B|\leq W+f(a)+f(b)$.
Both $f(t)$ and $tf(t)$ are nonincreasing for $2\leq t\leq n-3$.
For the latter assertion, when the terms are positive,
\[
\frac{(t+1)f(t+1)}{tf(t)}
=\frac{(t+1)(\ell-t)}{t(n-3-t)}\leq1,
\]
since $\ell\leq n/2\leq2(n-2)/3\leq t(n-2)/(t+1)$.
Terms with $t>\ell$ vanish. Consequently
\[
(a+b)(f(a)+f(b))
\leq4f(2)+2(n-3)f(2)=2(n-1)f(2),
\]
where the diagonal terms use monotonicity of $tf(t)$ and the
cross terms use $a,b\leq n-3$ and $f(a),f(b)\leq f(2)$.
It follows that
\begin{equation}\label{eq:k4-bound}
|\mathcal A||\mathcal B|\leq2(n-3)W+2(n-1)\binom{n-5}{\ell-2}.
\end{equation}
Set $Q=Q(n,4,\ell)$. Counting by the number of points
chosen from a disjoint two-set and four-set for $Q$, and
from the two three-element cores for $W$, yields
\begin{align*}
Q&=8\binom{n-6}{\ell-2}+16\binom{n-6}{\ell-3}+14\binom{n-6}{\ell-4}+6\binom{n-6}{\ell-5}+\binom{n-6}{\ell-6},\\
W&=9\binom{n-6}{\ell-2}+18\binom{n-6}{\ell-3}+15\binom{n-6}{\ell-4}+6\binom{n-6}{\ell-5}+\binom{n-6}{\ell-6}.
\end{align*}
Thus $8W\leq9Q$ and $8\binom{n-5}{\ell-2}\leq Q$.
Using \eqref{eq:k4-bound}, we obtain
\[
|\mathcal A||\mathcal B|\leq\frac{5n-14}2Q
<\left(\binom{n-2}2+1\right)Q,
\]
because $n^2-10n+22>0$ for $n\geq8$. This remaining case is
strict, completing the equality analysis as well.
\end{proof}

For $2\leq r\leq k$, write
\begin{equation}\label{eq:DrQr}
d_r=\binom{n-r}{k-r},\qquad
Q_r=\binom n\ell-\binom{n-r}\ell-\binom{n-k}\ell
       +\binom{n-k-2}\ell.
\end{equation}
In particular $Q_2=Q(n,k,\ell)$ and $F(n,k,\ell)=(d_2+1)Q_2$.

\begin{lemma}\label{lem:core-numerical}
Let $k\geq5$ and $3\leq r\leq s$, with $r+s=k+2$.
Then $(d_r+d_s)Q_r<(d_2+1)Q_2$.
\end{lemma}
\begin{proof}
The proportion of $(\ell-1)$-sets of an $(n-2)$-element ground
set meeting a fixed $k$-set is at least $k/(n-2)$: sample a
uniform point of a uniform $(\ell-1)$-set.
The two summands defining $Q_2$ are nonincreasing in their index.
Thus
\[
Q_2\geq\frac{2k}{n-2}\binom{n-2}{\ell-1},\qquad
\frac{Q_r}{Q_2}\leq1+\frac{(r-2)(n-2)}{2k},
\]
where the second inequality uses
$Q_r=Q_2+\sum_{i=3}^r\binom{n-i}{\ell-1}$.

For $r=3$, we have $s=k-1$, $d_s=n-k+1$ and
$d_3=(k-2)d_2/(n-2)$. Hence
\begin{align*}
&2k(d_2+1)-(d_3+n-k+1)(n+2k-2)\\
&\quad=\left(k+2-\frac{2k(k-2)}{n-2}\right)d_2
       +2k-(n-k+1)(n+2k-2)\\
&\quad>3d_2+2k-2(n-4)(n-1)>0.
\end{align*}
Indeed the coefficient of $d_2$ is greater than three,
$n-k+1\leq n-4$, $n+2k-2\leq2(n-1)$, and
$d_2\geq\binom{n-2}3$. Moreover
$3\binom{n-2}3>2(n-4)(n-1)$ for $n\geq10$, equivalently
$n^2-9n+10>0$. This proves the desired comparison for $r=3$.

Suppose $r\geq4$. Successive binomial ratios give
\[
\frac{d_s}{d_r}\leq2^{-(s-r)},\qquad
\frac{d_r}{d_2}=\prod_{j=2}^{r-1}\frac{k-j}{n-j}.
\]
For fixed $k,r$, the expression
\[
\frac{d_r}{d_2}\left(1+\frac{(r-2)(n-2)}{2k}\right)
\]
decreases with $n$: after expansion it is a positive constant times
a sum of products of positive decreasing reciprocal factors.
At $n=2k$, every factor $(k-j)/(2k-j)$ is less than $1/2$,
and the parenthesis is less than $r-1$. Therefore the expression
is less than $(r-1)/2^{r-2}$.
For $r\geq5$ this gives
\[
\frac{(d_r+d_s)Q_r}{(d_2+1)Q_2}
<\frac{r-1}{2^{r-3}}\leq1.
\]
For $r=4$ and $k\geq8$, use $1+d_s/d_4\leq5/4$ to obtain an
upper bound $15/16$. For $k=6,7$, respectively, the values
$d_4/d_2$ at $n=2k$ are $2/15$ and $5/33$. The corresponding
ratios are strictly less than
$2(2/15)3=4/5$ and $(3/2)(5/33)3=15/22$.
All cases are strict.
\end{proof}

\begin{lemma}\label{lem:two-part-product}
Let $k\geq5$. Suppose $R,S\subseteq[n]$ are disjoint,
$3\leq r:=|R|\leq s:=|S|$, and $r+s=k+2$.
Let $\mathcal A$ be the disjoint union of nonempty $k$-uniform
families $\mathcal A_R,\mathcal A_S$ whose members contain
$R,S$, respectively. Every $\ell$-uniform family $\mathcal B$
cross-intersecting with $\mathcal A$ satisfies
\[
|\mathcal A||\mathcal B|\leq(d_r+d_s)Q_r<F(n,k,\ell).
\]
\end{lemma}
\begin{proof}
Put $a=k-r=s-2$ and $b=r-2$. Then
$a\geq b\geq1$, $a\geq2$, and $k=a+b+2$.
Write
\[
E=\binom{n-r}\ell,
\quad E_0=\binom{n-k}\ell,\quad h=|\mathcal A_R|.
\]
In particular $|\mathcal A|\leq h+d_s$.
Choose $X\in\mathcal A_S$. Since $S\subseteq X$,
$|X\cap R|\leq r-2$. Inclusion--exclusion shows that there
are at most $Q_r$ members of $\mathcal B$ meeting $R$ and $X$:
the count increases with $|X\cap R|$ and has value $Q_r$ at
$|X\cap R|=r-2$.
Let $z$ be the number of members of $\mathcal B$ avoiding $R$.
These form a family cross-intersecting with
$\{A\setminus R:A\in\mathcal A_R\}$ on $[n]\setminus R$.
Meeting one fixed residual member, and applying
Lemma~\ref{lem:density}, respectively, give
\begin{equation}\label{eq:two-z}
z\leq E-E_0,\qquad
z\leq\frac{EL(d_r-h)}{d_rL+h},\qquad
L=\frac{a\ell}{(n-r)(n-k-\ell)}.
\end{equation}
Set
\[
w=\frac{d_s}{d_r},\quad c=1-\frac{E_0}E,\quad
\rho^2=\frac L{1+L}=\frac{a\ell}{(n-k)(n-r-\ell)},\quad
q=\frac{Q_r}E.
\]
We will establish
\begin{equation}\label{eq:two-comparisons}
q>\rho^2(1+w),\qquad
q>wc+\rho^2(1+w)(1-c).
\end{equation}
First,
\[
0<w=\frac{\binom{n-2}b}{\binom{n-2}a}\leq1,
\qquad wa\leq b.
\]
For the latter inequality, if $a=b$ it is equality; if $a>b$,
binomial monotonicity and $n\geq2a+2b+4$ give
$w\leq(b+1)/(n-b-2)\leq b/a$.

Count $Q_r$ using an $r$-set and a $k$-set whose intersection
has size $b$. Sets meeting this intersection yield
\[
Q_r\geq\binom n\ell-\binom{n-b}\ell
\geq b\binom{n-b}{\ell-1}\geq b\binom{n-r}{\ell-1}.
\]
Since $n-r-\ell\geq a+1$ and $n-k\geq k=a+b+2$,
\[
\frac q{\rho^2}\geq
\frac{b(n-k)(n-r-\ell)}{a(n-r-\ell+1)}
\geq\frac{b(a+b+2)(a+1)}{a(a+2)}.
\]
If $b\geq2$, this is greater than $b\geq2\geq1+w$.
If $b=1$, then
\[
w\leq\frac2{n-3}\leq\frac2{2a+3}
<\frac{2a+3}{a(a+2)},
\]
whereas the last lower bound for $q/\rho^2$ equals
$1+(2a+3)/(a(a+2))$. This proves the first inequality in
\eqref{eq:two-comparisons}.

Partitioning the same counted sets according to whether they
meet the common $b$-set gives the identity
\begin{equation}\label{eq:Qr-decomposition}
Q_r=\binom n\ell-\binom{n-b}\ell
 +\sum_{i=1}^{2}\sum_{j=1}^{a+2}\binom{n-b-i-j}{\ell-2}.
\end{equation}
The union bound gives $E-E_0\leq a\binom{n-r-1}{\ell-1}$,
so the first difference in \eqref{eq:Qr-decomposition} is at least
$b\binom{n-r-1}{\ell-1}\geq w(E-E_0)$.
The double sum is at least $2(a+2)\binom{n-k-2}{\ell-2}$.
Furthermore
\[
\rho^2 E_0
=\frac{a(n-k-1)}{(n-r-\ell)(\ell-1)}\binom{n-k-2}{\ell-2}
<a\binom{n-k-2}{\ell-2},
\]
because
\[
(n-r-\ell)(\ell-1)-(n-k-1)
=(\ell-2)(n-r-\ell-1)+a-1>0.
\]
Thus the double sum is strictly greater than
$2\rho^2E_0\geq(1+w)\rho^2E_0$. This proves the second
inequality in \eqref{eq:two-comparisons}.

If $h=d_r$, \eqref{eq:two-z} forces $z=0$, giving
$|\mathcal A||\mathcal B|\leq(d_r+d_s)Q_r$.
Otherwise put $x=h/d_r<1$. From \eqref{eq:two-z},
\[
\frac zE\leq\min\left\{c,\frac{L(1-x)}{L+x}\right\}.
\]
If $L\geq w$, then
\[
\frac{(x+w)z}{E(1-x)}
\leq\frac{L(x+w)}{L+x}
\leq\frac{L(1+w)}{1+L}<q.
\]
If $L<w$, the function $c(x+w)/(1-x)$ is increasing and
$L(x+w)/(L+x)$ is decreasing. They coincide at
$x=L(1-c)/(L+c)\in[0,1)$, with common value
$wc+\rho^2(1+w)(1-c)<q$.
Their minimum is therefore less than $q$ throughout $[0,1)$.
In either case $(h+d_s)z<(d_r-h)Q_r$, and consequently
\[
|\mathcal A||\mathcal B|
\leq(h+d_s)(Q_r+z)<(d_r+d_s)Q_r.
\]
Lemma~\ref{lem:core-numerical} completes the proof.
\end{proof}

\subsection{The two-core product theorem and its critical consequence}
The uniformities in the following theorem need not be ordered,
and neither family is assumed to be nontrivial.
\begin{theorem}\label{thm:two-core-product}
Let $k,\ell\geq3$, $n\geq2\max\{k,\ell\}$, and $n>k+\ell$.
Let $\mathcal A\subseteq\binom{[n]}k$ and
$\mathcal B\subseteq\binom{[n]}\ell$ be nonempty and
cross-intersecting. Suppose there are disjoint sets $R,S$ with
$2\leq|R|,|S|\leq k$ and $|R|+|S|\geq k+2$ such that every
member of $\mathcal A$ contains $R$ or $S$, and both parts are
nonempty. Then
\[
|\mathcal A||\mathcal B|
\leq\left(\binom{n-2}{k-2}+1\right)
\sum_{i=1}^{2}\left(\binom{n-i}{\ell-1}
                         -\binom{n-k-i}{\ell-1}\right).
\]
Equality holds if and only if, for disjoint
$P\in\binom{[n]}2$ and $X\in\binom{[n]}k$,
\begin{align*}
\mathcal A&=\{A\in\tbinom{[n]}k:P\subseteq A\}\cup\{X\},\\
\mathcal B&=\{B\in\tbinom{[n]}\ell:
 B\cap P\neq\emptyset,\ B\cap X\neq\emptyset\}.
\end{align*}
\end{theorem}
\begin{proof}
Let $R,S$ be the disjoint cores in the hypotheses.
If either part is a singleton $\{X\}$, choose a two-element
subset of the other core outside $X$; such a subset exists since
$|R|+|S|\geq k+2$. All other members contain that pair.
Lemma~\ref{asd} gives the required bound, with equality exactly
for $\mathcal C_k(R',X),\mathcal D_\ell(R',X)$ for the chosen
two-element set $R'$.

The cases $k=3,4$ are covered by Proposition~\ref{prop:small-critical}.
Suppose $k\geq5$ and neither part is a singleton. Neither core
has size $k$. Neither has size two, since that would force the
other to have size at least $k$. We may therefore shrink the
cores to sizes $3\leq r\leq s\leq k-1$ with $r+s=k+2$.
Every member still contains its assigned core; no $k$-set contains
both, so the two nonempty parts are retained.
Lemma~\ref{lem:two-part-product} gives a strict inequality in
this case. Conversely the stated extremal pair has product
$F(n,k,\ell)$ by Proposition~\ref{prop:two-point-construction}.
\end{proof}

\begin{corollary}\label{critically}
Let $\mathcal A\subseteq\binom{[n]}k$ and
$\mathcal B\subseteq\binom{[n]}\ell$ be nonempty families.
Suppose they are cross-intersecting and $\mathcal A$ is nontrivial
and critically shifted in the single-family sense. Then
\[
|\mathcal A||\mathcal B|\leq F(n,k,\ell).
\]
Equality holds if and only if, for some disjoint
$R\in\binom{[n]}2$ and $X\in\binom{[n]}k$,
\[
(\mathcal A,\mathcal B)=(\mathcal C_k(R,X),\mathcal D_\ell(R,X)).
\]
The equality characterization is within the class satisfying the
hypotheses, up to a simultaneous relabeling.
\end{corollary}
\begin{proof}
Lemma~\ref{lem:two-cores} supplies disjoint cores satisfying all
the hypotheses of Theorem~\ref{thm:two-core-product}. That theorem
gives the product bound and the equality characterization. The
canonical extremal pair is critically shifted by
Proposition~\ref{prop:two-point-construction}.
\end{proof}

\begin{lemma}[Reversing a shift of the two-point construction]
\label{lem:inverse-two-point}
Let $a,b\geq3$ and $n>a+b$. Suppose that
$(\mathcal F,\mathcal G)$ is a saturated cross-intersecting pair
with uniformities $a,b$, and that for some $i<j$,
\[
 s_{ij}(\mathcal F)=\mathcal C_a(R,X),\qquad
 s_{ij}(\mathcal G)=\mathcal D_b(R,X),
\]
where $|R|=2$, $|X|=a$ and $R\cap X=\emptyset$.
Then $(\mathcal F,\mathcal G)$ is itself a two-point construction,
up to a simultaneous relabeling.
\end{lemma}
\begin{proof}
Write $\mathcal F'=\mathcal C_a(R,X)$. If $i,j\notin R$,
the full $R$-star is unchanged by the shift and its inverse;
only the exceptional member can change. If $i,j\in R$,
every member of $\mathcal F'$ contains both or neither of
$i,j$, so $\mathcal F=\mathcal F'$. In either case saturation
determines $\mathcal G$ as the required full companion.

The case $j\in R$, $i\notin R$ is impossible: a full
$R$-star is not fixed by $s_{ij}$, even after adjoining $X$,
whereas every image of $s_{ij}$ is fixed by $s_{ij}$.
It remains that $R=\{i,p\}$ and $j\notin R$.
Here $j\notin X$, since otherwise the shift of $X$ contains
$i$ but not $p$ and is absent from $\mathcal F'$.

Put $V=[n]\setminus\{i,j,p\}$. For each
$H\in\binom V{a-2}$, exactly one of
$\{i,p\}\cup H$ and $\{j,p\}\cup H$ belongs to
$\mathcal F$. Color $H$ by that choice, $i$ or $j$.
Similarly, for every $K\in\binom V{b-1}$ meeting $X$,
exactly one of $\{i\}\cup K$ and $\{j\}\cup K$
belongs to $\mathcal G$; color $K$ accordingly.
Cross-intersection forces $H$ and $K$ to have the same color
whenever $H\cap K=\emptyset$.

If $H,H'\in\binom V{a-2}$ differ in one element, then
$|H\cup H'|=a-1<|X|$ and
$|V\setminus(H\cup H')|=n-a-2\geq b-1$.
Thus there is a $(b-1)$-set $K$ disjoint from both and
meeting $X$, so $H,H'$ have the same color.
The graph on $\binom V{a-2}$ joining sets that differ in one
element is connected. Moreover every eligible $K$ has a
disjoint $(a-2)$-set in $V$. All colors are therefore equal.

Consequently $\mathcal F$ is either the full
$\{i,p\}$-star together with $X$, or the full
$\{j,p\}$-star together with $X$. Saturation again gives
the claimed form of $\mathcal G$.
\end{proof}

\begin{corollary}\label{cor:critical-pair-product}
Let $n\geq2\ell\geq2k\geq6$, with $n>k+\ell$.
A saturated, nontrivial, critically shifted cross-intersecting
pair satisfies
\[
|\mathcal A||\mathcal B|\leq
\max\{M_1(n,k,\ell),M_2(n,k,\ell)\}.
\]
Equality holds precisely when $M_2(n,k,\ell)\geq M_1(n,k,\ell)$
and, for disjoint $R\in\binom{[n]}2$ and
$Y\in\binom{[n]}\ell$,
\[
 (\mathcal A,\mathcal B)
 =(\mathcal D_k(R,Y),\mathcal C_\ell(R,Y)).
\]
This characterization is within the class in the hypotheses.
\end{corollary}
\begin{proof}
By Lemma~\ref{lem:critical-separation}, one family is critically
shifted in the single-family sense. If it is $\mathcal A$,
Corollary~\ref{critically} and Lemma~\ref{lem:orientation} give
$|\mathcal A||\mathcal B|\leq F(n,k,\ell)<M_1(n,k,\ell)$.
If it is $\mathcal B$, interchange the uniformities to obtain
$|\mathcal A||\mathcal B|\leq F(n,\ell,k)=M_2(n,k,\ell)$.
The equality characterization in Corollary~\ref{critically}
then gives exactly the asserted pair. Conversely that pair
has product $M_2$.
\end{proof}

Thus the critical alternative in Lemma~\ref{lem:normalization}
satisfies Theorem~\ref{mthm}. The shifted alternative is proved
in Section~\ref{sec:shifted-product} using the restricted-lex
comparison of Theorem~\ref{sec123:cck}.

\section{The product of shifted nontrivial cross-intersecting families}
\label{sec:shifted-product}

Throughout this section we allow the two uniformities to occur in
either order. Unless otherwise stated, assume
\begin{equation}\label{eq:shifted-range}
 k,\ell\geq3,\qquad n\geq2\max\{k,\ell\},\qquad n>k+\ell,
\end{equation}
and put
\begin{align*}
 S&=\binom{n-1}{k-1}-\binom{n-\ell-1}{k-1},
 &J&=\binom{n-2}{k-2},\\
 T&=\binom{n-1}{\ell-1}-\binom{n-k-1}{\ell-1},
 &K&=\binom{n-2}{\ell-2}.
\end{align*}
In particular, $(S+1)(T+1)=M_1(n,k,\ell)$ is symmetric in
$k,\ell$. We first prove the oriented estimate
\begin{equation}\label{eq:oriented-target}
 |\mathcal A||\mathcal B|
 \leq\max\{(S+1)(T+1),(2S-J)(K+2)\}
 \quad\hbox{if }|\mathcal A|\geq S+1.
\end{equation}
The second term will then be compared with the two bounds in the
introduction. The reduction to these two products uses the
restricted-lex comparison from Section~\ref{sec:shifting}.

\subsection{Removing layers beyond the second}

\begin{lemma}\label{lem:meeting-normalized-matching}
Let $W$ be a nonempty proper subset of an $N$-element set, and let
$\mathcal U_j=\{H:|H|=j,\ H\cap W\neq\emptyset\}$.
If $1\leq d\leq e\leq N$ and $\mathcal H\subseteq\mathcal U_d$,
then its upper $e$-shadow inside $\mathcal U_e$ satisfies
\[
 \frac{|\{G\in\mathcal U_e:H\subseteq G
                    \text{ for some }H\in\mathcal H\}|}{|\mathcal U_e|}
 \geq\frac{|\mathcal H|}{|\mathcal U_d|}.
\]
\end{lemma}
\begin{proof}
It suffices to couple uniform members of two consecutive ranks
so that the smaller one is contained in the larger one. Write
$b=|W|$. For a uniform member of $\mathcal U_j$, the number $X_j$
of its elements in $W$ has probability masses proportional to
$\binom bi\binom{N-b}{j-i}$, for $i\geq1$.
On their common support, the ratios of the masses of $X_{j+1}$
to those of $X_j$ are proportional to
$(N-b-j+i)/(j+1-i)$, which increases with $i$.
The ratios of the masses of $X_{j+1}-1$ to those of $X_j$ are
proportional to $(b-i)/(i+1)$, which decreases with $i$.
The endpoints of the supports give the same comparisons.
The elementary monotone-likelihood-ratio criterion therefore gives
$X_j\leq_{\rm st}X_{j+1}\leq_{\rm st}X_j+1$.
(The criterion follows by noting that the difference of the two
mass functions changes sign at most once and summing its tails.)

Use the same uniform quantile to couple these counts. Then
$X_j\leq X_{j+1}\leq X_j+1$. Independently order $W$ and its
complement uniformly at random, and take the first $X_j$ and
$j-X_j$ elements of the two orders. The resulting sets are nested
and have the required uniform marginals. Membership of the smaller
set in $\mathcal H$ forces membership of the larger set in its
upper shadow. Iteration proves the assertion.
\end{proof}

\begin{lemma}\label{lem:remove-late-layers}
Let $\mathcal L=\mathcal L_\ell([n],k,m)$ under
\eqref{eq:shifted-range}. If $m\geq2S-J$, then
\[
 m\,|\overline N_\ell(\mathcal L)|\leq(2S-J)(K+2).
\]
\end{lemma}
\begin{proof}
The capacity of the layer with minimum element $i$ is
$c_i=\binom{n-i}{k-1}-\binom{n-\ell-1}{k-1}$.
The first two layers have total size $c_1+c_2=2S-J$ and
companion size $K+2$, by Proposition~\ref{prop:restricted-lex-count}.
For $3\leq r\leq\ell$, write
$m=P+q$, where $P=\sum_{i<r}c_i$ and $0\leq q\leq c_r$.
Set $b=\ell-r+1$ and $Q=\binom{n-r+1}{b}+r-1$.
The same proposition gives
$|\overline N_\ell(\mathcal L)|=Q-d$, where $d$ is the size of the
disjointness neighborhood of a family $\mathcal H$ of $q$
$(k-1)$-sets in $U=[r+1,n]$, all meeting a fixed $b$-set
$W\subseteq U$. Taking the complement of each neighborhood member
in $U$ gives
\[
\{U\setminus B:B\in N_b^U(\mathcal H)\}
=\{G\in\tbinom U{n-r-b}:H\subseteq G
                         \text{ for some }H\in\mathcal H\}.
\]
Thus complementation is a bijection onto the upper
$(n-r-b)$-shadow of $\mathcal H$, every member of which meets $W$.
Applying
Lemma~\ref{lem:meeting-normalized-matching} to ranks $k-1$ and
$n-r-b$ gives
\begin{equation}\label{eq:late-layer-slope}
 d\geq q\,\frac{\binom{n-r}b-1}{c_r}.
\end{equation}
Here $n-r-b=n-\ell-1\geq k$, and exactly one $(n-r-b)$-set
avoids $W$, which explains the subtraction of one.

We verify $(\binom{n-r}b-1)/c_r>Q/P$.
Since $P>(r-1)c_r$, it suffices that
\begin{equation}\label{eq:late-binomial}
 (r-2)\binom{n-r}b-\binom{n-r}{b-1}\geq2(r-1).
\end{equation}
We have $n-r\geq2b+r-2$. For $b=1$, \eqref{eq:late-binomial}
holds except when $r=3,n-r\leq4$.
For $r=3,b\geq2$, its left side is at least
$\binom{2b+1}b-\binom{2b+1}{b-1}\geq5$.
For $r\geq4,b\geq2$, it is at least
$(r-3)\binom{r+2}2\geq2(r-1)$.
The only admissible exception is $(n,k,\ell)=(7,3,3)$ and $r=3$;
then $P=19,Q=7,c_r=3$, and
$(\binom{n-r}b-1)/c_r=1>7/19$ directly.
Consequently, for $q>0$, \eqref{eq:late-layer-slope} implies
\[
 (P+q)(Q-d)\leq(P+q)(Q-qQ/P)\leq PQ.
\]
Apply this successively to the last occupied layer, stopping after
the second full layer. This proves the lemma.
\end{proof}

\subsection{The first block of the second layer}

\begin{lemma}\label{lem:first-block-binomial}
Set $D=\binom{n-3}{k-2}$ and
$v=\binom{n-3}{\ell-1}-\binom{n-k-1}{\ell-1}$.
Under \eqref{eq:shifted-range},
\begin{equation}\label{eq:first-block-slope}
 v(S+1)\geq(D-1)(T+1)
\end{equation}
except possibly for
\[
 (n,k,\ell)=(8,3,4),(9,3,4),(10,3,5),(12,3,6),(10,4,5).
\]
\end{lemma}
\begin{proof}
The following verification includes both orientations of the
uniformities. Put
\begin{gather*}
 a=k-2,\quad c=\ell-2,\quad
 \tau=n-k-\ell+1\geq2,\\
 E=\binom{n-3}c,
 \quad Z=\frac{\binom{n-k-1}c}{\binom{n-3}c},\quad R=\tau(1-Z).
\end{gather*}
Thus $n-3\geq2\max\{a,c\}+1$ and $D=\binom{n-3}a$.
The hockey-stick identity gives
\begin{align*}
 \frac vE&=\frac{a+R}{c+1},&
 \frac{S-D+2}{D}
 &=\frac{n-2}{c+\tau+1}+\frac{c+R}{a+1}+\frac2D,\\
 T&=v+E+\binom{n-2}c.
\end{align*}
Hence \eqref{eq:first-block-slope} is equivalent to
\begin{equation}\label{eq:initial-normalized}
 \frac{a+R}{c+1}
 \left(\frac{n-2}{c+\tau+1}+\frac{c+R}{a+1}+\frac2D\right)
 \geq\left(1-\frac1D\right)
 \left(1+\frac{n-2}{a+\tau+1}+\frac1E\right).
\end{equation}
Write
\[
 G(R)=\frac{a+R}{c+1}
 \left(\frac{n-2}{c+\tau+1}+\frac{c+R}{a+1}\right)
       -1-\frac{n-2}{a+\tau+1}.
\]
This increases with $R\geq0$. If $a\leq c$, then $D\leq E$,
and $G(R)\geq0$ suffices for \eqref{eq:initial-normalized}:
the omitted correction, multiplied by $D$, is
$2v/E+1+(n-2)/(a+\tau+1)-(D-1)/E>0$.

For fixed $a,c$, $R$ is nondecreasing in $\tau$.
Indeed, $Z=\prod_{i=1}^a(\tau+i)/(\tau+c+i)$, and
$Z^{-1}-1\geq\sum_{i=1}^a c/(\tau+i)$ gives
\[
 R'=1-Z-\tau Z\sum_{i=1}^a
       \frac{c}{(\tau+i)(\tau+c+i)}\geq0.
\]
For fixed $\tau$, $R$ also increases with $a,c$, as follows
from this product and its symmetric expression.
In particular, $R\geq6/5$ if $a,c\geq2$.
Direct simplification gives
\begin{equation}\label{eq:G-at-one}
 G(1)=\frac{(n-2)(n-1)(a-c)}
 {(c+1)(c+\tau+1)(a+\tau+1)}.
\end{equation}
This settles $a=c\geq2$. If $a>c$ and $R\geq1$, then
$G(R)>1/(c+1)>1/E$, which also implies
\eqref{eq:initial-normalized}, even after discarding its positive
terms involving $1/D$. For $c=1,a\geq3$, we have $\tau\geq a$
and $R\geq a^2/(2a+1)>1$. Thus this deals with $a>c$ except
$(a,c)=(2,1)$.

Suppose now that $c>a$. Set $d=c-a$, so $\tau\geq d+1$.
If $a\geq3$, put $\tau_0=d+1$. The last three factors of $Z$
at $\tau=\tau_0$ have product
\[
 \prod_{i=a-2}^{a}\frac{\tau_0+i}{\tau_0+c+i}
 =\frac{c^2-1}{2(4c^2-1)}<\frac18.
\]
Each preceding factor is less than $1/2$, so $Z(\tau_0)<2^{-a}$.
By the monotonicity of $R$ and the inequality $2^a\geq2a+1$,
\[
R\geq R(\tau_0)>(d+1)(1-2^{-a})\geq R_0:=\frac{2a(d+1)}{2a+1}.
\]
For $a=1$ or $2$, the corresponding exact lower bounds at
$\tau=c-a+1$ are
\[
 R_0=\frac{c^2}{2c+1}\quad(a=1),\qquad
 R_0=\frac{(c-1)(3c+1)}{4c+2}\quad(a=2).
\]
Here is an explicit polynomial check that these bounds imply
$G(R_0)\geq0$ for $a\geq3$, for $a=1,c\geq6$, and for
$a=2,c\geq4$. Clearing the positive denominator in $G(R_0)$
leaves, using $n=a+c+\tau+3$,
\begin{align*}
 F(a,c,\tau,R_0)
 ={}&(a+R_0)\big[(a+c+\tau+1)(a+1)+(c+R_0)(c+\tau+1)\big]
                         (a+\tau+1)\\
 &-(c+1)(a+1)(c+\tau+1)(2a+c+2\tau+2).
\end{align*}
In each row of the following table, substitute the indicated
nonnegative variables. Multiplication by the displayed factor
gives a polynomial all of whose nonzero coefficients are positive;
its constant coefficient is recorded in the last column.
This verifies the required signs by expansion.
\[
\begin{array}{c|c|c|c|r}
 a&c&\tau&\text{factor multiplying }F&\text{constant}\\ \hline
 x+3&a+y+1&y+2+z&(2a+1)^2&1120\\
 1&y+6&y+6+z&(2y+13)^2&46956\\
 2&y+4&y+3+z&4(2y+9)^2&20520
\end{array}
\]
It remains to check a finite list of pairs of small uniformities.
Let $H(n)=v(S+1)-(D-1)(T+1)$.
For $k=\ell=3$, $H(n)=0$ identically.
For the other pairs, the table lists the coefficients, in increasing
powers of $x=n-n_0$, of the polynomial $d_0H(n_0+x)$.
All are positive, so the indicated ranges are covered.
\[
\begin{array}{c|c|r|r|l}
 k&\ell&n_0&d_0&\text{coefficients}\\ \hline
3&4&10&2&(30,53,14,1)\\
3&5&11&3&(84,369,149,21,1)\\
3&6&13&24&(13176,20922,8129,1347,103,3)\\
3&7&14&120&(126960,284156,136996,29285,3220,179,4)\\
4&3&8&1&(53,37,10,1)\\
4&5&11&12&(9756,16020,7583,1590,157,6)
\end{array}
\]
The admissible values below these starting points are precisely
the five triples stated in the lemma.
\end{proof}

\begin{lemma}\label{lem:first-block-product}
If $1\leq q\leq D=\binom{n-3}{k-2}$ and
$\mathcal L=\mathcal L_\ell([n],k,S+q)$, then the following
inequality holds, with equality if and only if $q=1$:
\[
 (S+q)|\overline N_\ell(\mathcal L)|\leq(S+1)(T+1).
\]
\end{lemma}
\begin{proof}
Put $f(q)=\Psi(n-2,k-1,\ell-1;q)$, in the notation of
Section~\ref{sec:shifting}. Proposition~\ref{prop:restricted-lex-count}
gives
\[
 |\overline N_\ell(\mathcal L)|=\binom{n-1}{\ell-1}+1-f(q),
 \quad f(1)=\binom{n-k-1}{\ell-1},\quad
 f(D)=\binom{n-3}{\ell-1}.
\]
The real-binomial form of Kruskal--Katona due to Lov\'asz
\cite{Frankl1984} gives the chord estimate
\begin{equation}\label{eq:first-block-chord}
 f(q)\geq f(1)+\frac{q-1}{D-1}\big(f(D)-f(1)\big).
\end{equation}
Here is the concavity detail for this application. Write
$q=\binom xp$, where $p=n-k-1>\ell-1=b$ and $p\leq x\leq n-3$.
Kruskal--Katona gives $f(q)\geq\binom xb$.
The latter is concave as a function of $\binom xp$ on this range:
if $g(x)=\binom xb$ and $P(x)=\prod_{j=b}^{p-1}(x-j)$, then
$\binom xp=(b!/p!)gP$ and
\[
 \frac{d\binom xp}{dg}=\frac{b!}{p!}
 \left(P+\frac{g}{g'}P'\right)
\]
is increasing, since $P',P''\geq0$ and
$g/g'=(\sum_{j=0}^{b-1}(x-j)^{-1})^{-1}$ is increasing.
Thus its graph lies above the chord joining the two endpoints,
as asserted in \eqref{eq:first-block-chord}.

Except for the five triples in Lemma~\ref{lem:first-block-binomial},
put $\lambda=(f(D)-f(1))/(D-1)$. That lemma gives
$\lambda(S+1)\geq T+1$. Since $q\geq1$,
\[
 (S+q)|\overline N_\ell(\mathcal L)|
 \leq(S+q)(T+1-\lambda(q-1))
 \leq(S+1)(T+1).
\]
For $q>1$, subtracting $(S+1)(T+1)$ from the middle expression
gives at most $-\lambda(q-1)^2<0$.
For the five remaining triples, the exact binomial-expansion
formula of Theorem~\ref{cokk} applies directly.
The following exact values show that the maximum is attained
only at $q=1$; here
$V(q)=(S+q)(\binom{n-1}{\ell-1}+1-f(q))$:
\[
\begin{array}{c|c|c|r|r|r}
n&k&\ell&D&V(1)&\max_{2\leq q\leq D}V(q)\\ \hline
8&3&4&5&608&598\\
9&3&4&6&1081&1036\\
10&3&5&7&3472&3404\\
12&3&6&9&18722&18198\\
10&4&5&21&9882&9676
\end{array}
\]
This completes the proof.
\end{proof}

\subsection{The remaining blocks of the second layer}

For $3\leq r\leq\ell+1$, define
\begin{equation}\label{eq:block-endpoints}
 A_r=S+\sum_{i=3}^r\binom{n-i}{k-2},\quad
 B_r=K+1+\binom{n-r}{\ell-r+1},\quad P_r=A_rB_r.
\end{equation}
In particular,
$P_3\leq(S+1)(T+1)$ by Lemma~\ref{lem:first-block-product}, and
$P_{\ell+1}=(2S-J)(K+2)$.

\begin{samepage}
\begin{lemma}\label{lem:block-parameters}
For $3\leq r\leq\ell$, write $b=\ell-r+1$ and $t=n-k-\ell>0$,
and put
\[
 D=\binom{n-r-1}{k-2},\quad E=\binom{n-r-1}b,
 \quad\alpha=\frac{A_r}{D},\quad\beta=\frac{B_{r+1}}E.
\]
Then both of the following alternatives hold:
\begin{align}
 &\alpha\geq\frac{(k-2)b}{(n-r-1)t}\quad\hbox{or}\quad\alpha>\beta+1,
 \label{eq:block-density-alternative}\\
 &\alpha\geq\frac{b(n-r-1+b+t)}{(n-r-1)t}\quad\hbox{or}\quad\alpha>\beta.
 \label{eq:block-difference-alternative}
\end{align}
\end{lemma}
\end{samepage}
\begin{proof}
The parameter assumptions give
$n-r-1=k-2+b+t\geq2b+r-3$,
$k-2+t\geq b+r-3$, and $n\geq2k$.
The first $r$ terms in
$S=\sum_{i=2}^{\ell+1}\binom{n-i}{k-2}$ and the $r-2$ additional
terms in $A_r$ are at least $D$, so $\alpha\geq2r-2$.
We will repeatedly use
\begin{equation}\label{eq:block-r-three}
 \alpha\geq
 \frac{(n-2)(n-3)}{(b+t+2)(b+t+1)}
 +\frac{2(n-3)}{b+t+1}+1
 \geq\left(1+\frac{n-2}{b+t+2}\right)^2
 \quad(r=3).
\end{equation}
For $r\geq4$, put $F=\binom{n-2}{k-2}/D$.
The first four terms of $S$ and the additional terms indexed by
$3,4$ give
\begin{equation}\label{eq:block-alpha-F}
 \alpha\geq\frac{21}{8}F.
\end{equation}
Indeed, the three consecutive ratios beginning with
$\binom{n-2}{k-2}$ are at least $1/2$ by $n\geq2k$.
Moreover, direct cancellation of factorials gives
\begin{equation}\label{eq:block-K-F}
 \frac{K}{EF}
 =\frac{\prod_{j=1}^{r-1}(b+t+j)}
 {\big(\prod_{j=1}^{r-3}(b+j)\big)(n-\ell-1)(n-\ell)}
 \leq\left(1+\frac{t}{b+1}\right)^{r-1}.
\end{equation}

First suppose $\alpha<(k-2)b/((n-r-1)t)$. Then
$t<b/\alpha\leq b/(2r-2)$.
For $r=3$, this gives $b\geq4t+1\geq5$ and $b-3t\geq2$.
Thus \eqref{eq:block-r-three} gives $\alpha\geq25/4$,
whereas
\[
 \beta+1=\frac{\binom{n-r+1}b+\binom{n-r}b+1}{\binom{n-r-1}b}
 <4+2+\frac1{10}<\frac{25}{4}.
\]
For $r\geq4$, \eqref{eq:block-K-F} and
$(r-1)t/(b+1)<(r-1)/(2r-2)=1/2$ imply $K/E<2F$;
here one may use $(1+z)^u\leq(1-uz)^{-1}$ for $uz<1$.
Also $t<b/6$, $k-2\geq b+1-t$, and
\[
 F\geq\prod_{j=1}^3\left(1+\frac{k-2}{b+t+j}\right)
 \geq\left(\frac{2b+4}{b+t+3}\right)^3
 >\left(\frac85\right)^3.
\]
Since $\binom{n-r}b/E<2$ and $E\geq4$,
$\beta+1<2F+9/4$. Consequently
$\alpha-(\beta+1)>(5/8)F-9/4>0$.
This proves \eqref{eq:block-density-alternative}.

For the second alternative, suppose
$\alpha<b(n-r-1+b+t)/((n-r-1)t)$. Since $n-r-1\geq2b$,
\begin{equation}\label{eq:block-second-small-t}
 t<\frac{3b}{2\alpha-1}\leq\frac{3b}{4r-5}.
\end{equation}
If $r=3$, then $b\geq3$ and $t<3b/7$. Equation
\eqref{eq:block-r-three} yields
$\alpha>(1+14/11)^2=625/121$, while
$\beta<4+1+1/20=101/20<625/121$.
If $r\geq4$, then $b\geq4$, $E\geq\binom94=126$, and
\eqref{eq:block-K-F} gives
$K/(EF)<\exp\left(\frac{3(r-1)}{4r-5}\right) \leq e^{9/11}<16/7$.
Furthermore, \eqref{eq:block-second-small-t} implies the integer
inequality $b-3t\geq1$. Hence $F\geq(3/2)^3$ by the same
three-factor estimate. It follows that
\[
 \alpha-\beta>
 \left(\frac{21}{8}-\frac{16}{7}\right)F-\frac{127}{126}
 \geq\frac{19}{56}\frac{27}{8}-\frac{127}{126}>0.
\]
This proves \eqref{eq:block-difference-alternative} as well.
\end{proof}

\begin{lemma}\label{lem:block-endpoint-reduction}
For $3\leq r\leq\ell$ and $A_r\leq q\leq A_{r+1}$,
\[
 q\,|\overline N_\ell(\mathcal L_\ell([n],k,q))|
 \leq\max\{P_r,P_{r+1}\}.
\]
\end{lemma}
\begin{proof}
The first layer is
$\mathcal S=\{A:|A|=k,\ 1\in A,\ A\cap[2,\ell+1]\neq\emptyset\}$,
of size $S$. The consecutive blocks in the second layer are
\[
 \mathcal F_i=\{\{2,i\}\cup H:H\in\tbinom{[i+1,n]}{k-2}\},
 \qquad 3\leq i\leq\ell+1.
\]
Using the notation of Lemma~\ref{lem:block-parameters}, the
restricted-lex family of size $q=A_r+x$ consists of $\mathcal S$,
the full blocks $\mathcal F_3,\ldots,\mathcal F_r$, and
$\{\{2,r+1\}\cup H:H\in\mathcal H\}$, where
$\mathcal H\subseteq\binom{[r+2,n]}{k-2}$ has size $x$.

We count its companion by a disjoint partition.
An $\ell$-set avoiding $1$ and meeting all of $\mathcal S$ must be
$[2,\ell+1]$: if it also avoids $j\in[2,\ell+1]$, there are
at least $n-\ell-2\geq k-2$ further points with which to extend
$\{1,j\}$ to a disjoint member of $\mathcal S$.
All $K$ sets containing $\{1,2\}$ are in the companion.
A companion set containing $1$ but avoiding $2$ must contain
$[3,r]$. If not, take the least omitted $i\in[3,r]$.
The set then contains $[3,i-1]$, so the tail of $\mathcal F_i$
has at least $n-\ell-2\geq k-2$ available points outside it,
giving a disjoint member of that block. There are
$\binom{n-r-1}{b-1}$ such companion sets also containing $r+1$.
The remaining ones have the form $\{1\}\cup[3,r]\cup G$,
where $G\in\binom{[r+2,n]}b$ meets every member of $\mathcal H$.
Writing $\mathcal G$ for this family, we obtain exactly
\begin{equation}\label{eq:block-exact-companion}
 |\overline N_\ell(\mathcal L_\ell([n],k,q))|
 =K+1+\binom{n-r-1}{b-1}+|\mathcal G|=B_{r+1}+|\mathcal G|.
\end{equation}

Put
\[
u=\frac{|\mathcal H|}{D},\qquad
v=\frac{|\mathcal G|}{E},\qquad
L=\frac{(k-2)b}{(n-r-1)t}.
\]
Lemma~\ref{lem:density} gives $u+v+uv/L\leq1$.
If $\alpha\geq L$, set $w=\max\{\alpha,\beta\}\geq L$;
then
$(\alpha+u)(\beta+v) \leq\alpha\beta+w(u+v)+uv\leq\alpha\beta+w$.
Otherwise Lemma~\ref{lem:block-parameters} gives
$\alpha>\beta+1$. Since $v\leq1-u$,
$(\alpha+u)(\beta+v) \leq(\alpha+u)(\beta+1-u)\leq\alpha(\beta+1)$.
Multiplying by $DE$, and using $B_r=B_{r+1}+E$ and
$A_{r+1}=A_r+D$, proves the desired endpoint bound.
\end{proof}

\begin{lemma}\label{lem:endpoint-sequence}
The products in \eqref{eq:block-endpoints} satisfy
\[
 \max_{3\leq r\leq\ell+1}P_r=\max\{P_3,P_{\ell+1}\}.
\]
\end{lemma}
\begin{proof}
With the notation of Lemma~\ref{lem:block-parameters},
$P_{r+1}-P_r=D B_{r+1}-A_rE=DE(\beta-\alpha)$.
We show that a nonnegative difference is followed by a nonnegative
difference. For $3\leq r\leq\ell-1$, the parameters at the next
index obey, by direct cancellation of binomial coefficients,
\[
 \alpha'=\frac{n-r-1}{b+t}(\alpha+1),\qquad
 \beta'=\frac{n-r-1}b\beta-1.
\]
If $\beta\geq\alpha$, the second alternative in
Lemma~\ref{lem:block-parameters} forces
$\alpha\geq b(n-r-1+b+t)/((n-r-1)t)$. Consequently
\[
 \beta'-\alpha'
 =\frac{n-r-1}b(\beta-\alpha)
   +\frac{(n-r-1)t}{b(b+t)}\alpha-1-\frac{n-r-1}{b+t}\geq0.
\]
Thus the sequence of differences can change sign only from negative
to nonnegative. The maximum of the finite sequence is therefore at
one of its two ends.
\end{proof}

\subsection{The final comparison and the shifted-pair theorem}

\begin{lemma}\label{lem:last-endpoint-comparison}
Under \eqref{eq:shifted-range},
\[
 (2S-J)(K+2)
 \begin{cases}
 <(S+1)(T+1),&k\geq\ell,\\
 <M_2(n,k,\ell),&k<\ell.
 \end{cases}
\]
\end{lemma}
\begin{proof}
Suppose first that $k\geq\ell$. The estimate used in
Lemma~\ref{lem:orientation}, with the uniformities interchanged,
gives $T\geq2K+2$. Also $S\leq\ell J$ and $K+2\geq\ell$.
Therefore
$(S+1)(T+1)-(2S-J)(K+2) \geq J(K+2)-S+2K+3>0$.

Now let $k<\ell$ and put $D_0=\binom{n-\ell-2}{k-2}$.
Pascal's identity gives
$M_2(n,k,\ell)=(2S-J+D_0)(K+1)$, so it suffices to show
$D_0(K+1)>2S-J$. We prove the stronger inequality
$D_0K\geq2S-J$. Since $D_0>0$, this also proves the
strict inequality in the statement. Direct cancellation gives
\begin{equation}\label{eq:last-ratio}
 \frac{D_0K}{J}
 =\binom{n-k}{\ell}
   \frac{\ell(\ell-1)}{(n-\ell)(n-\ell-1)}.
\end{equation}
The ratio on the left increases with $n$: both $D_0$ and
$K/J$ do so. If $\ell\geq k+2$, its value at $n=2\ell$ is
at least $\binom{\ell+2}2\geq2\ell$, whereas
$(2S-J)/J\leq2\ell-1$.

It remains that $k=\ell-1$. The ratio in
\eqref{eq:last-ratio}, divided by $n$, is nondecreasing:
$D_0/n$ and $K/J$ are nondecreasing separately. Its value at
$n=2\ell$ is $(\ell+1)/(2\ell)$. Thus, for $\ell\geq6$,
\[
 \frac{D_0K}{J}\geq\frac{n(\ell+1)}{2\ell}
 \geq\frac{2n}{\ell-2}
 >\frac{2(n-1)}{k-1}-1\geq\frac{2S-J}{J}.
\]
For $(k,\ell)=(4,5)$, the left side of
\eqref{eq:last-ratio} is $(n-7)(n-8)(n-4)/6$, while
$(2S-J)/J\leq(2n-5)/3$. Their difference, on writing $n=10+x$,
is $1+16x/3+11x^2/6+x^3/6\geq0$.
For $(k,\ell)=(3,4)$, $2S-J=7n-26$ and
$D_0K=(n-6)(n-2)(n-3)/2$; their difference at $n=8+x$ is
$19x+13x^2/2+x^3/2\geq0$. These exhaust the remaining cases.
\end{proof}

\begin{theorem}\label{thm:shifted-pair-product}
Let $n\geq2\ell\geq2k\geq6$ and $n>k+\ell$.
If $\mathcal A\subseteq\binom{[n]}k$ and
$\mathcal B\subseteq\binom{[n]}\ell$ are nonempty, nontrivial,
shifted and cross-intersecting, then
\[
 |\mathcal A||\mathcal B|
 \leq\max\{M_1(n,k,\ell),M_2(n,k,\ell)\}.
\]
\end{theorem}
\begin{proof}
We first establish \eqref{eq:oriented-target} without ordering
$k,\ell$. Nontriviality and shiftedness give
$\mathcal A\subseteq\mathcal B(n,k,\ell)$ by
Lemma~\ref{lem:restricted-universe}. If $m=|\mathcal A|\geq S+1$,
Theorem~\ref{sec123:cck} therefore gives
\begin{equation}\label{eq:shifted-restricted-reduction}
 |\mathcal B|\leq|\overline N_\ell(\mathcal A)|
 \leq|\overline N_\ell(\mathcal L_\ell([n],k,m))|.
\end{equation}
For $S+1\leq m\leq A_3$, apply
Lemma~\ref{lem:first-block-product}.
For $A_3\leq m\leq A_{\ell+1}=2S-J$, apply
Lemmas~\ref{lem:block-endpoint-reduction} and
\ref{lem:endpoint-sequence}; the two terminal products are bounded
by $(S+1)(T+1)$ and $(2S-J)(K+2)$.
For $m\geq2S-J$, apply Lemma~\ref{lem:remove-late-layers}.
These intervals cover every $m\geq S+1$, proving
\eqref{eq:oriented-target}.

Return now to $k\leq\ell$. If
$|\mathcal A|\leq S+1$ and $|\mathcal B|\leq T+1$, the product
is at most $M_1$. Otherwise apply \eqref{eq:oriented-target}
to a family exceeding its corresponding threshold, interchanging
the two families if necessary. Lemma~\ref{lem:last-endpoint-comparison}
bounds the second endpoint by $M_2$ in the original orientation
and by $M_1$ in the reversed orientation. This proves the theorem.
\end{proof}

\begin{corollary}\label{cor:shifted-product-equality}
Under the hypotheses of Theorem~\ref{thm:shifted-pair-product},
equality is possible only if $M_1(n,k,\ell)\geq M_2(n,k,\ell)$
and
$|\mathcal A|=S+1,\  |\mathcal B|=T+1$.
In that case, up to isomorphism, the pair has one of the forms
in Proposition~\ref{prop:threshold-equality}.
\end{corollary}
\begin{proof}
By Lemma~\ref{lem:first-block-product}, the first block has
product strictly less than $M_1$ at every size greater than
$S+1$. In particular, $P_3<M_1$.
Lemma~\ref{lem:last-endpoint-comparison} gives a strict bound
for the other endpoint as well: it is less than $M_2$ in the
original orientation and less than $M_1$ in the reversed
orientation.
Lemmas~\ref{lem:block-endpoint-reduction},
\ref{lem:endpoint-sequence} and
\ref{lem:remove-late-layers} now show that the oriented
product estimate is strict whenever the selected family has
size greater than its threshold. This argument applies in
either orientation.

Thus equality requires $|\mathcal A|\leq S+1$ and
$|\mathcal B|\leq T+1$. Their product then forces both
equalities and $M_1\geq M_2$.
Let $\mathcal B^*=\overline N_\ell(\mathcal A)$.
Since $\mathcal B\subseteq\mathcal B^*$ and $\mathcal B$ is
nonempty and nontrivial, so is $\mathcal B^*$.
Proposition~\ref{prop:threshold-equality} gives
$|\mathcal B^*|\leq T+1=|\mathcal B|$.
Hence $\mathcal B=\mathcal B^*$, and
Proposition~\ref{prop:threshold-equality} gives the stated
classification. Conversely, these pairs have
product $M_1$.
\end{proof}

\begin{proof}[Proof of Theorem~\ref{mthm}.]
Take an extremal nontrivial cross-intersecting pair.
Lemma~\ref{lem:normalization} transforms it, preserving both
cardinalities, into a saturated pair that is either shifted or
critically shifted. Theorem~\ref{thm:shifted-pair-product} and
Corollary~\ref{cor:critical-pair-product} give the claimed bound.
The two constructions in the introduction attain $M_1$ and
$M_2$, respectively, so the bound is sharp.

Now let $(\mathcal A,\mathcal B)$ attain equality. Every pair
in its normalization sequence is again extremal and saturated.
If the terminal pair is shifted,
Corollary~\ref{cor:shifted-product-equality} gives
$M_1\geq M_2$ and terminal sizes $S+1,T+1$.
These are also the sizes of the original families.
Apply Proposition~\ref{prop:threshold-equality} directly to
the original saturated pair. It gives precisely (i) or (iii).
Thus no assertion that shifting preserves isomorphism is needed.

If the terminal pair is critically shifted, equality in
Corollary~\ref{cor:critical-pair-product} forces $M_2\geq M_1$
and the form in (ii). Reverse the normalization sequence,
applying Lemma~\ref{lem:inverse-two-point} at each step with
$(a,b)=(\ell,k)$. The original pair therefore also has the
form in (ii).

Conversely, the pairs in (i) and (ii) attain their respective
products by the counts in the introduction. In (iii), both
families have size $3n-8$, so their product is $M_1(n,3,3)$.
Lemma~\ref{lem:orientation} gives
$M_2(n,3,3)=F(n,3,3)<M_1(n,3,3)$.
All the displayed pairs are nontrivial and cross-intersecting,
which proves sufficiency and completes the classification.
\end{proof}

\section*{ AI-use declaration}
Generative AI tools were  used to assist with the proofs of the inequalities in Lemmas 4.2–4.8 and with basic error checking and language polishing of the manuscript. We take full responsibility for all content of this paper.

\setlength{\bibsep}{2pt plus 1pt minus 1pt}

\end{document}